\documentclass[10pt, leqno]{amsart}

\usepackage{
  color,
  amsmath,
  amsthm,
  amssymb,
  tikz,
  fancyhdr,
  enumitem,
  setspace,
  mathdots,
  multirow,
  upgreek,
}

\usepackage[all]{xy}

\usepackage{listings}
\lstdefinelanguage{Sage}[]{Python}
{morekeywords={False,sage,True},sensitive=true}
\definecolor{dblackcolor}{rgb}{0.0,0.0,0.0}
\definecolor{dbluecolor}{rgb}{0.01,0.02,0.7}
\definecolor{dgreencolor}{rgb}{0.2,0.4,0.0}
\definecolor{dgraycolor}{rgb}{0.30,0.3,0.30}

\newenvironment{red}{\relax\color{red}}{\relax}
\newenvironment{blue}{\relax\color{blue}}{\hspace*{.5ex}\relax}

\newcommand{\ber}{\begin{red}}
\newcommand{\er}{\end{red}}
\newcommand{\beb}{\begin{blue}}
\newcommand{\eb}{\end{blue}}
\def\CC{{\Bbb C}}
\def\FF{{\Bbb F}}
\def\GG{{\Bbb G}}

\def\ZZ{{\Bbb Z}}
\def\LL{{\Bbb L}}

\theoremstyle{bfupright head,upright body}
\newtheorem{res}{}[section]             \newtheorem*{res*}{}
               \newtheorem*{bfhpg*}{}

\theoremstyle{fixed bf head,slanted body}
\newtheorem{thm}[res]{Theorem}          \newtheorem*{thm*}{Theorem}
\newtheorem{prp}[res]{Proposition}      \newtheorem*{prp*}{Proposition}
            \newtheorem*{lem*}{Lemma}

\theoremstyle{fixed bf head,upright body}
       \newtheorem*{dfn*}{Definition}
           \newtheorem*{rmk*}{Remark}
\newtheorem{conj}[res]{Conjecture}           \newtheorem*{conj*}{Conjecture}

\theoremstyle{numbered paragraph}

\newlength{\thmlistleft}        
\newlength{\thmlistright}       
\newlength{\thmlistpartopsep}   
\newlength{\thmlisttopsep}      
\newlength{\thmlistparsep}      
\newlength{\thmlistitemsep}     

\newcounter{prt}
  {\end{list}}%

\newenvironment{prf*}[1][Proof]{%
  \begin{proof}[\bf #1]
    \setcounter{equation}{0}
    }
  {\end{proof}
}
\renewcommand{\eqref}[1]{\pgref{eq:#1}}
\newcommand{\pgref}[1]{(\ref{#1})}

\def\urltilda{\kern -.15em\lower .7ex\hbox{\~{}}\kern .04em}

\numberwithin{equation}{res}

\newcommand{\Res}{\operatorname{res}}

\def\OOO{{\mathcal O}}

\begin{document}

\title{Some branching formulas for Kac--Moody Lie algebras}

\author[K.-H. Lee]{Kyu-Hwan Lee}

\address{University of Connecticut, Storrs, CT~06269, U.S.A.}

\email{khlee@math.uconn.edu}

\urladdr{http://www.math.uconn.edu/\urltilda khlee}

\thanks{K.-H.L.\ was partially supported by a grant from the Simons Foundation (\#318706).}

\author[J. Weyman]{Jerzy Weyman}

\address{University of Connecticut, Storrs, CT~06269, U.S.A.}

\email{jerzy.weyman@uconn.edu}

\urladdr{http://www.math.uconn.edu/\urltilda weyman}

\thanks{J.W.\ was
  partly supported by NSF DMS grant 1802067.}

\date{\today}

\keywords{Kac--Moody algebras, branching formulas}

\subjclass[2010]{13C99; 13H10.}

\begin{abstract}
 In this paper we give some branching rules for the fundamental representations
of Kac--Moody Lie algebras associated to $T$-shaped graphs. These formulas are useful to describe generators 
of the generic rings  for free resolutions of length three described in \cite{JWm18}.
We also make some conjectures about the generic rings.

\end{abstract}

\maketitle

\thispagestyle{empty}

\section{Introduction}
\label{sec:introduction}

\noindent
Let $T_{p,q,r}$ be a $T$-shaped graph defined as follows.
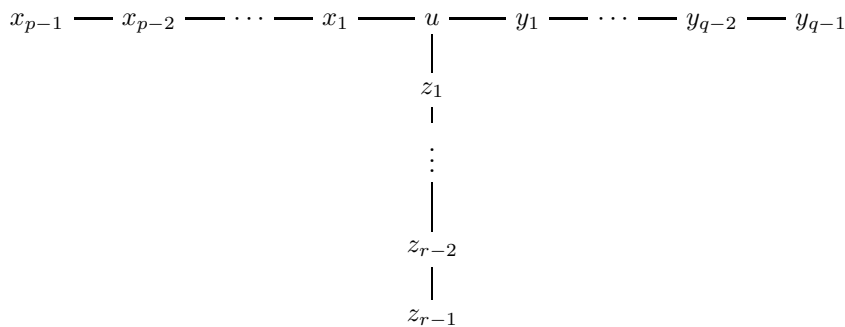
\[
\raisebox{1em}{\xymatrixcolsep{1.2 pc}\xymatrixrowsep{1pc}\xymatrix{ 
{x_{p-1}} \ar@{-}[r]  & {x_{p-2}} \ar@{-}[r]  
&{\cdots} \ar@{-}[r] &  {x_1} \ar@{-}[r] & {u}\ar@{-}[r]   & {y_1}\ar@{-}[r]& {\cdots}\ar@{-}[r]& {y_{q-2}}\ar@{-}[r]& {y_{q-1}}\ar@{-}[]
\\ &&&& {z_1} \ar@{-}[u] \\ &&&& {\vdots} \ar@{-}[u]\\ &&&& {z_{r-2}} \ar@{-}[u]\\ &&&& {z_{r-1}} \ar@{-}[u] }}
\]
There are two other notations for vertices that we will be using.
Sometimes the indices will be indexed by the set
$\lbrace 0,1,\ldots, p-1, 1',\ldots ,(q-1)', 1'',\ldots ,(r-1)''\rbrace$ corresponding to the vertices $u, x_1 ,\ldots ,x_{p-1}$, $y_1,\ldots ,y_{q-1}$, $z_1,\ldots ,z_{r-1}$ respectively.
Sometimes we denote vertices by natural numbers from $[1,p+q+r-2]$, in the order listed above.

The main result of \cite{JWm18} associates to every graph $T_{p,q,r}$ the format of free resolutions of length three over commutative rings, and constructs a particular generic ring ${\hat R}_{gen}$  for the resolutions of that format which deforms to a ring ${\hat R}_{spec}$ on which the Kac--Moody Lie algebra ${\mathfrak g}(T_{p,q,r})$ associated to $T_{p,q,r}$ acts.
The correspondence is as follows. For the free resolutions of format $(r_1,r_1+r_2, r_2+r_3, r_3)$, i.e. those with the ranks of differentials given by $(r_1, r_2, r_3)$, we have
$$(p,q,r)= (r_1+1, r_2-1, r_3+1).$$
In this context it is important to consider the grading on ${\mathfrak g}(T_{p,q,r})$ associated to the simple root corresponding to the vertex $z_1$; more precisely
$${\mathfrak g}(T_{p,q,r})=\bigoplus_{i\in\ZZ}{\mathfrak g}_i(T_{p,q,r})$$
where ${\mathfrak g}_i(T_{p,q,r})$ is the span of roots where the simple root corresponding to vertex $z_1$ occurs with multiplicity $i$.

The generators of the ring ${\hat R}_{gen}$ involve the fundamental representations of the algebra ${\mathfrak g}(T_{p,q,r})$. Especially important are those corresponding to extremal vertices $x_{p-1}$, $y_{q-1}$, $z_{r-1}$. Thus it is important to find the restrictions of these representations to the Lie subalgebra of ${\mathfrak g}(T_{p,q,r})$ corresponding to the graph we obtain from $T_{p,q,r}$ by omitting the vertex $z_1$.

In this note we calculate these restrictions. Note that the graph $T_{p,q,r}$ with the node $z_1$ removed is a disjoint union of two Dynkin graphs of types $A_{p+q-1}$ and of type $A_{r-2}$. We denote these Lie algebras by $\mathfrak {sl}(F_1)$ and $\mathfrak {sl}(F_3)$ where $F_1$ is a vector space over $\CC$ of dimension $p+q$ and $F_3$ is a vector space over $\CC$ of dimension $r-1$. We also denote by $\mathfrak {gl}(F_1)$ and $\mathfrak {gl}(F_3)$ the corresponding general linear Lie algebras.

This paper is organized as follows. In Section \ref{section:results} we recall the results of \cite{JWm18}, and  list the formulas for the restrictions of representations we need  for the tables in the following sections, and state some conjectures. We employ the Bourbaki notation, but always present the graph $T_{p,q,r}$ as above.
In Section \ref{section:dn} we deal with the graphs of type $D_n$. In Section \ref{section:e6} we deal with $E_6$, in Section \ref{section:e7} with $E_7$ and in Section \ref{section:e8} with $E_8$. 
In the last section we illustrate the usefulness of the tables by observing some general patterns from the tables and discussing some plausible conjectures.

\subsection*{Acknowledgments} We thank the anonymous referee for useful comments and suggestions. J.W.\ would like to thank  
Lars~W. Christensen and Oana Veliche for helpful discussions
regarding the material of this paper. K.-H.L.\ would like to thank Ben Salisbury and Travis Scrimshaw for their help with SageMath.

\section{The free resolutions of length $3$ and Kac--Moody algebras related to $T$-shaped graphs}
\label{section:results}

The problem motivating calculations presented in this note has to do with generic rings for finite free resolutions. For a partition $\lambda$, we denote by $S_\lambda$ the Schur functor associated with $\lambda$. In particular, $S_n$ denotes the symmetric power. The exterior power  will be denoted by $\bigwedge^n$. 

We consider the free acyclic complexes  $\FF_\bullet$ (i.e complexes whose only (possible) nonzero homology group is $H_0(\FF_\bullet)$) of the form
$$\FF_\bullet:\ 0\rightarrow F_3\rightarrow F_2\rightarrow F_1\rightarrow F_0$$
over  commutative Noetherian rings $R$,
with $rank\ F_i=f_i$ ($0\le i\le 3$), $rank (d_i)=r_i$ ($1\le i\le 3$). The quadruple $(f_0, f_1, f_2, f_3)$ is {\it the format} of the complex $\FF_\bullet$. 
We always have $f_i=r_i+r_{i+1}$ ($0\le i\le 3$).

For the resolutions of such format $(f_0, f_1, f_2, f_3)$ we say that a pair $(R_{gen}, \FF^{gen}_\bullet)$ where $R_{gen}$ is a commutative ring and $\FF^{gen}_\bullet$ is a  free  acyclic complex over $R_{gen}$ is {\it a generic resolution} of this format
if two conditions are satisfied:

\begin{enumerate}
\item The complex $\FF^{gen}_\bullet$ is acyclic over $R_{gen}$;
\item For every acyclic free complex $\GG_\bullet$ over a Noetherian ring $S$ there exists a ring homomorphism $\phi:R_{gen}\rightarrow S$ such that
$$\GG_\bullet =\FF^{gen}_\bullet\otimes_{R_{gen}}S.$$
\end{enumerate}

Of particular interest is whether the ring $R_{gen}$ is Noetherian, because it can be shown quite easily that a non-Noetherian (non-unique) generic pair always exists.
For complexes of length $2$ the existence of the pairs $(R_{gen}, \FF^{gen}_\bullet)$ was established by Hochster (\cite{H75}). He also proved that this generic ring is Noetherian.

In \cite{JWm18} the particular  generic rings ${\hat R}_{gen}$ were constructed for all formats $(f_0, f_1, f_2, f_3)$. 
The main result of \cite{JWm18} is (Theorem 9.1 there)

\begin{thm} For every format $(f_0, f_1, f_2, f_3)$ there exists a generic pair
$$({\hat R}_{gen}, \FF^{gen}_\bullet):=((j_{gen})_* \OOO_{X_{gen}\setminus D_3}, \FF^a_\bullet\otimes_{R_a} (j_{gen})_* \OOO_{X_{gen}\setminus D_3}).$$
The generic ring ${\hat R}_{gen}$ deforms to a ring ${\hat R}_{spec}$ which carries a multiplicity free action of ${\mathfrak g}(T_{p,q,r})\times \mathfrak {gl}(F_2)\times \mathfrak {gl}(F_0)$, where $f_3=r-1, f_2=q+r, f_1=p+q, r_1=p-1$.
If the defect Lie algebra $\LL(r_1+1, F_3, F_1)$ is finite dimensional, then the generic ring ${\hat R}_{gen}$  is Noetherian. 
\end{thm}

It is natural to call the format $(f_0, f_1, f_2, f_3)$ Dynkin  if the diagram $T_{p,q,r}$ is a Dynkin graph.
The rings  ${\hat R}_{gen}$ have an explicit decomposition to representations of $\prod_{i=0}^3 GL(F_i)$ which we now describe.

For a sextuple $\mu =(a,b,c,\alpha, \beta, \gamma )$ with $a\ge 0$, where $\alpha$ is a partition with $\le r_3-1$ parts, $\beta$ is a partition with $\le r_2$ parts and $\gamma$ is a partition with $\le r_1-1$ parts,
we define 
\begin{align*}
\sigma (\mu)&:= (a-b+c+\alpha_1 ,\ldots ,a-b+c+\alpha_{r_3-1}, a-b+c),\\
\tau(\mu)&:=(c+\gamma_1 ,\ldots ,c+\gamma_{r_1-1}, c, c-b, c-b-\beta_{r_2-1},\ldots ,c-b-\beta_1),\\
\theta(\mu)&:= (b-c+\beta_1 ,\ldots, b-c+\beta_{r_2 -1}, b-c, -a+b-c, -a+b-c-\alpha_{r_3-1},\ldots, -a+b-c-\alpha_1),\\
\phi(\mu)&:= (0^{f_0-r_1}, -c, -c-\gamma_{r_1-1},\ldots ,-c-\gamma_1).
\end{align*}

The Buchsbaum--Eisenbud  multipliers ring (see \cite{JWm18}) has a decomposition
$$R_a=\oplus_{\mu} S_{\phi(\mu)} F_0\otimes S_{\tau(\mu)} F_1\otimes S_{\theta(\mu )}F_2\otimes S_{\sigma(\mu)} F_3.$$

For given $\alpha, \beta, t$ we define the weight $\lambda (\sigma, \tau, t)$ of ${\mathfrak g}(T_{p,q,r})$ as follows.
We label the vertices of $T_{p,q,r}$ on the third arm by the coefficients of fundamental weights in $\sigma$, i.e.
$$\lambda_{p+q+i}=\sigma_{r-1-i}-\sigma_{r-i}$$
for $i=1,\ldots, r-2$.
We also label the vertices at the center and the first two arms by coefficients of fundamental weights in $\tau$,
i.e.
$$\lambda_0=\tau_p-\tau_{p+1},$$
$$ \lambda_i=\tau_{p-i}-\tau_{p-i+1}$$
for $1\le i\le p-1$, and
$$\lambda_i=\tau_i-\tau_{i+1}$$
for $i=p+1,\ldots, p+q-1$.
Finally, we put
$$\lambda_{p+q}=a.$$
Finally we set $t:=a+1$.

Then we have (\cite{JWm18}, Proposition 9.3)

\begin{prp}\label{prp-decrgen} 
We have an $\mathfrak {gl}(F_0)\times \mathfrak {gl}(F_2)\times {\mathfrak g}(T_{p,q,r})$ decomposition of a deformation ${\hat R}_{spec}$ of ${\hat R}_{gen}$
$${\hat R}_{spec}=\oplus_{\mu} S_{\phi(\mu)} F_0\otimes S_{\theta(\mu )}F_2\otimes V_{\lambda (\sigma(\mu), \tau(\mu), a)}$$
where $V_{\lambda}$ is the irreducible lowest weight module of weight $\lambda$ for ${\mathfrak g}(T_{p,q,r})$.
The module $V_\lambda$ is the highest weight representation for the opposite Borel subalgebra.
It is also irreducible.
The ring ${\hat R}_{gen}$ has a decomposition as a representation of $\prod_{i=0}^3 GL(F_i)$
$${\hat R}_{gen}=\oplus_{\mu} S_{\phi(\mu)} F_0\otimes S_{\theta(\mu )}F_2\otimes \Res (V_{\lambda (\sigma(\mu), \tau(\mu), a)})$$
where $\Res$ denotes the restriction from ${\mathfrak g}(T_{p,q,r})$ to $\mathfrak {gl}(F_1)\times \mathfrak {gl}(F_3)$.
\end{prp}

Let us denote by $W(\mu)$ the isotypic component $S_{\phi(\mu)} F_0\otimes S_{\theta(\mu )}F_2\otimes \Res(V_{\lambda (\sigma(\mu), \tau(\mu), a)})$.
The representation $W(\mu)$ is equipped with the grading 
$$W(\mu)=\oplus_{i\ge 0}W(\mu)(i)$$
induced by the grading
$$V_{\lambda (\sigma(\mu), \tau(\mu), a)}=\oplus_{i\ge 0} V_{\lambda (\sigma(\mu), \tau(\mu), a)}(i).$$
Notice that the graded summand $W(\mu)(0)$ gives the corresponding summand in $R_a$.

One also has the description of the generators of ${\hat R}_{gen}$ (\cite{JWm18}, Proposition 10.1).

\begin{prp} \label{prp-aa}
The generators of the semigroup of weights in ${\hat R}_{gen}$ are as follows:
\begin{enumerate}
\item $\alpha=(1^i), \beta=\gamma=a=b=c=0$, for $1\le i\le r_3-1$,
\item $a=1$, $\alpha=\beta=\gamma=b=c=0$,
\item $\beta=(1^j), \alpha=\gamma=a=b=c=0$, for $1\le j\le r_2-1$,
\item $b=1$, $\alpha=\beta=\gamma=a=c=0$,
\item $\gamma=(1^k), \alpha=\beta=a=b=c=0$, for $1\le k\le r_1-1$,
\item $c=1$, $\alpha=\beta=\gamma=a=b=0$.
\end{enumerate}
The rings ${\hat R}_{spec}$ and ${\hat R}_{gen}$ are generated by the corresponding representations $W(\mu)$.
\end{prp}

In fact we expect that the ring ${\hat R}_{gen}$ is generated by much smaller set of representations (see \cite{JWm18} for motivation).

\begin{conj}\label{conj:mingen} The ring ${\hat R}_{gen}$ is generated by the six representations $W(\mu)$ corresponding to the following weights.
\begin{enumerate}
\item $\alpha=(1), \beta=\gamma=a=b=c=0$, corresponding to $i=1$ in Proposition \ref{prp-aa} (1),
\item $a=1$, $\alpha=\beta=\gamma=b=c=0$,
\item $\beta=(1), \alpha=\gamma=a=b=c=0$, corresponding to $j=1$ in Proposition \ref{prp-aa} (3),
\item $b=1$, $\alpha=\beta=\gamma=a=c=0$,
\item $\gamma=(1), \alpha=\beta=a=b=c=0$,corresponding to $k=1$ in Proposition \ref{prp-aa} (5),
\item $c=1$, $\alpha=\beta=\gamma=a=b=0$.
\end{enumerate}
\end{conj}

For the formats with $r_1=1$, i.e. resolutions of cyclic modules, we actually have even stronger expectation.

\begin{conj}\label{conj:mingen-1} Let us assume that $r_1=1$. The ring ${\hat R}_{gen}$ is generated by the four representations $W(\mu)$ corresponding to the following weights.
\begin{enumerate}
\item $\alpha=(1), \beta=\gamma=a=b=c=0$, corresponding to $i=1$ in Proposition \ref{prp-aa} (1),
\item $a=1$, $\alpha=\beta=\gamma=b=c=0$,
\item $\beta=(1), \alpha=\gamma=a=b=c=0$, corresponding to $j=1$ in Proposition \ref{prp-aa} (3),
\item $b=1$, $\alpha=\beta=\gamma=a=c=0$,
\item $c=1$, $\alpha=\beta=\gamma=a=b=0$.
\end{enumerate}
Moreover, the first representation is redundant if $r_3=1$, and the second one is redundant if $r_3>1$. The last representation is just a variable $a_1$, so it is completely understood.
\end{conj}

Particularly important are three critical representations. The first one is the one corresponding to $\alpha=(1), \beta=\gamma=a=b=c=0$ (or to $a=1$, $\alpha=\beta=\gamma=b=c=0$ if $r_3=1$). We denote it by $W(d_3)$. Similarly, the representation $W(d_2)$ is the one corresponding to $\beta=(1), \alpha=\gamma=a=b=c=0$.
Finally, $W(d_1)$ is the representation corresponding to $\gamma =(1)$, $\alpha=\beta=a=b=c=0$ (in the case $r_1=1$ we replace it by $W(a_2)$ i.e.
$b=1$, $\alpha=\beta=\gamma=a=c=0$).

All representations $W(\mu)$ have a grading
$$W(\mu)=\oplus_{i\ge 0}W_i(\mu).$$
So in order to understand the generators of ${\hat R}_{gen}$ in the case $r_1=1$ we need to understand the restrictions to $\mathfrak {gl}(F_3)\times \mathfrak {gl}(F_1)$ of three fundamental representations of ${\mathfrak g}(T_{p,q,r})$ corresponding to extremal nodes of the diagram $T_{p,q,r}$. This is carried out case by case in the following sections.

We end with some remarks about the tables. For each situation we list the representations in graded components of representations $V(\omega_{x_p})$, $V(\omega_{y_q})$, $V(\omega_{z_r})$. The bold-faced first column of a table shows degrees in the graded decomposition, and the middle and last column with numeric values have decomposition multiplicities.    Note that each table actually serves two formats, as $p$ and $q$ can be swapped. But this amounts to changing $F_1$ to $F_1^*$.
Some of the representations are very big, so in some cases we list just half of the representation. The other half then can be read off by duality.


\section{The type $D_n$}
\label{section:dn}

The first format is $(1,n,n,1)$. The graph $T_{p,q,r}$ is

\[ 
\raisebox{1em}{\xymatrixcolsep{1.2 pc}\xymatrixrowsep{1pc}\xymatrix{ 
{x_1} \ar@{-}[r]  &{u}
\ar@{-}[r] &
  {y_1} \ar@{-}[r] & {y_2}
\ar@{-}[r]  & {\cdots}
\ar@{-}[r] & {y_{n-3}} 
\\ & {z_1} \ar@{-}[u]  }}
\]
We number the vertices as follows:
\[
\raisebox{1em}{\xymatrixcolsep{2 pc}\xymatrixrowsep{1.5pc}\xymatrix{ 
*{\circ} \ar@{-}[r]^<{n}  &*{\circ}
\ar@{-}[r]^<{n-2}  &
  *{\circ} \ar@{-}[r]^<{n-3} & *{\ \cdots \ }
\ar@{-}[r]_<{}  & *{\circ}
\ar@{-}[]^<{1}
\\ & *{\bullet} \ar@{-}[u]^<{n-1}  }}
\]

The Lie algebra ${\mathfrak g}(T_{p,q,r})$ is ${\mathfrak g}(D_n) = \mathfrak {so}(2n)= \mathfrak {so}(F_1\oplus F^*_1)$ ($dim\ F_1=n$),  with the grading
$${\mathfrak g}(D_n)= {\mathfrak g}_{-1}\oplus {\mathfrak g}_0\oplus {\mathfrak g}_1$$
where ${\mathfrak g}_0= \mathfrak {sl}(F_3)\times \mathfrak {sl}(F_1)\times\CC =  \mathfrak {sl}(F_1)\times\CC$  and ${\mathfrak g}_1=F_3^*\otimes\bigwedge^2 F_1$.
The orthogonal space is $U:=F_1\oplus F_1^*$.

It is not difficult to see that

$$ V(\omega_1)= F_1^*\oplus F_3^*\otimes F_1,$$
$$V(\omega_{n-1} )=\oplus_{k=0}^{n\over 2} S_{1-k}F_3^*\otimes\bigwedge^{2k}F_1,$$
$$ V(\omega_{n})= \oplus_{k=0}^{n\over 2} S_{-k}F_3\otimes\bigwedge^{2k+1}F_1.$$

The next format is $(1,4,n,n-3)$. The graph $T_{p,q,r}$ with the distinguished root $z_1$ and the labeling of the vertices are as follows:
\[  
\raisebox{1em}{\xymatrixrowsep{1pc}\xymatrixcolsep{1pc}\xymatrix{ 
{x_1} \ar@{-}[r]  & {u} \ar@{-}[r]  
&{y_1} \ar@{-}[] \\ &  {z_1} \ar@{-}[u] \\ & {\vdots}\ar@{-}[u] \\  & {z_{n-4}} \ar@{-}[u]
\\ & {z_{n-3}} \ar@{-}[u]  }}
\hspace*{2.5cm}
\raisebox{1em}{\xymatrixrowsep{1.5 pc}\xymatrixcolsep{2 pc}\xymatrix{ 
*{\circ} \ar@{-}[r]^<{n-1}  &*{\circ}
\ar@{-}[r]^<{n-2}  &
  *{\circ} \ar@{-}[]^<{n} \\ & *{\bullet}
\ar@{-}[u]_<{n-3} \\ & {\vdots}
\ar@{-}[u]^<{} 
\\ & *{\circ} \ar@{-}[u]_<{2}\\ & *{\circ} \ar@{-}[u]_<{1} }}
\]


The corresponding Lie algebra is 
$${\mathfrak g}(T_{p,q,r})=\mathfrak {so}(2n).$$ 
The orthogonal space in question is
$$U:= F_3\oplus \bigwedge^2 F_1\oplus F_3^*,$$
with $dim\ F_3=n-3$ and $dim\ F_1=4$.
We have
$$V(\omega_1)=F_3\oplus\bigwedge^2 F_1\oplus F_3^*\otimes\bigwedge^4 F_1,$$
$$V(\omega_{n-1})=\bigwedge^{even}(F_3)\otimes F_1\oplus \bigwedge^{odd}(F_3)\otimes F_1^*,$$
$$V(\omega_n)= \bigwedge^{even}(F_3^*)\otimes F_1\oplus \bigwedge^{odd}(F_3^*)\otimes F_1^*.$$

\bigskip\bigskip

Finally we have the format $(n-3,n,4,1)$. The graph $T_{p,q,r}$ is

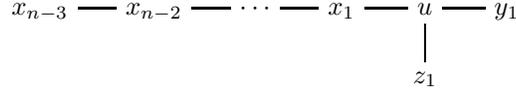
\[ 
\raisebox{1em}{\xymatrixcolsep{1.2 pc}\xymatrixrowsep{1pc}\xymatrix{ 
{x_{n-3}} \ar@{-}[r]  & {x_{n-2}} \ar@{-}[r]  
&{\cdots} \ar@{-}[r] &  {x_1} \ar@{-}[r] & {u}\ar@{-}[r]   & {y_{1}} 
\\ &&&& {z_1} \ar@{-}[u]  }}
\]
We number the vertices as follows:
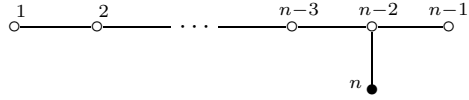
\[
\raisebox{1em}{\xymatrixcolsep{2 pc}\xymatrixrowsep{1.5pc}\xymatrix{ 
*{\circ} \ar@{-}[r]^<{1}  &*{\circ}
\ar@{-}[r]^<{2}  &
  *{\ \cdots \ } \ar@{-}[r]^<{} & *{\circ}
\ar@{-}[r]^<{n-3}  & *{\circ}
\ar@{-}[r]^<{n-2} & *{\circ}
\ar@{-}[]^<{n-1}
\\ &&&& *{\bullet} \ar@{-}[u]^<{n}  }}
\]
The orthogonal space is $U=F_1\oplus F_1^*.$ This case can be obtained from the first format $(1,n,n,1)$ with $F_1$ and $F_1^*$ swapped.





\section{The type $E_6$}
\label{section:e6}

\subsection{ $E_6$ graded by $\alpha_5$} This format is $(1,5,6,2)$.

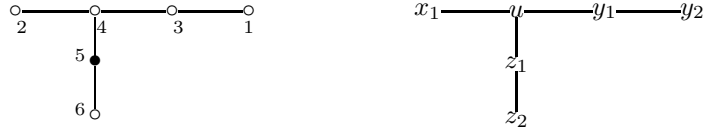
\[
\raisebox{1em}{\xymatrix@R=3ex{ 
*{\circ}<3pt> \ar@{-}[r]_<{2}  &*{\circ}<3pt>
\ar@{-}[r]_<{4}  &
  *{\circ}<3pt> \ar@{-}[r]_<{3}   &*{\circ}<3pt> \ar@{}[]_<{1}
\\ & *{\bullet}<3pt> \ar@{-}[u]^<{5} \\ & *{\circ}<3pt> \ar@{-}[u]^<{6}  }}
\hspace{2cm}
\raisebox{1em}{\xymatrix@R=3ex{ 
*{x_1}<3pt> \ar@{-}[r]_<{}  &*{u}<3pt>
\ar@{-}[r]_<{}  &
  *{y_1}<3pt> \ar@{-}[r]_<{}   &*{y_2}<3pt> \ar@{}[]_<{}
\\ & *{z_1}<3pt> \ar@{-}[u]^<{} \\ & *{z_2}<3pt> \ar@{-}[u]^<{}  }}
\]

\subsubsection{ $V(\omega_1)$} 
This representation is of dimension 27 and has 4 graded components.

\begin{center}

\begin{tabular}{|c||l|ll|c|}
\hline
{\bf 0} & $F_1^*$ &(0,0)&(0,0,0,0,-1)&1 \\ \hline
{\bf 1} & $F_3^*\otimes F_1$ &(1,0)&(1,0,0,0,0) &1 \\ \hline
{\bf 2} & $\bigwedge^2 F_3^*\otimes\bigwedge^3F_1$ &(1,1)&(1,1,1,0,0) &1 \\ \hline
{\bf 3} & $S_{2,1}F_3^*\otimes\bigwedge^5 F_1$ &(2,1)&(1,1,1,1,1) &1  \\ \hline
\end{tabular}

\end{center}

\subsubsection{ $V(\omega_6)$}
This representation is of dimension 27 and has 4 graded components.

\begin{center}

\begin{tabular}{|c||l|ll|c|}
\hline
{\bf 0} & $F_3$ &(0,-1)&(0,0,0,0,0)&1 \\ \hline
{\bf 1} & $\bigwedge^2F_1$ &(0,0)&(1,1,0,0,0) &1 \\ \hline
{\bf 2} & $F_3^*\otimes\bigwedge^4F_1$ &(1,0)&(1,1,1,1,0) &1 \\ \hline
{\bf 3} & $\bigwedge^2F_3^*\otimes S_{2,1,1,1,1}F_1$ &(1,1)&(2,1,1,1,1) &1  \\ \hline
\end{tabular}

\end{center}

\subsubsection{ $V(\omega_2)$}
This representation is of dimension 78 and has 5 graded components.
\begin{center}

\begin{tabular}{|c||l|ll|c||l|ll|c|}
\hline
{\bf 0} & $F_1$ &(0,0)&(1,0,0,0,0)&1 &&&& \\ \hline
{\bf 1} & $F_3^*\otimes \bigwedge^3 F_1$ &(1,0)&(1,1,1,0,0) &1 &&&& \\ \hline
{\bf 2} & $S_2F_3^*\otimes\bigwedge^5 F_1$ &(2,0)&(1,1,1,1,1) &1  & $\bigwedge^2 F_3^*\otimes \bigwedge^5 F_1$ &(1,1)&(1,1,1,1,1) &1  \\ \hline
 & $\bigwedge^2 F_3^*\otimes S_{2,1^3}F_1$ &(1,1)&(2,1,1,1,0) &1 &&&& \\ \hline
{\bf 3} & $S_{2,1}F_3^*\otimes S_{2^2,1^3}F_1$ &(2,1)&(2,2,1,1,1) &1  &&&& \\ \hline
{\bf 4} & $S_{2,2}F_3^*\otimes S_{2^4,1}F_1$ &(2,2)&(2,2,2,2,1) &1 &&&&  \\ \hline
\end{tabular}
\end{center}

\medskip 


\subsection{ $E_6$ graded by $\alpha_2$} This format is $(2,6,5,1)$.

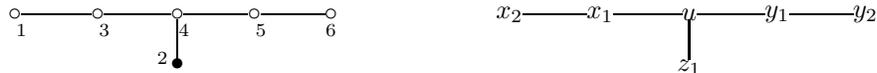
\[
\raisebox{1em}{\xymatrix@R=3ex{ *{\circ}<3pt> \ar@{-}[r]_<{1}  &
*{\circ}<3pt> \ar@{-}[r]_<{3}  &*{\circ}<3pt>
\ar@{-}[r]_<{4}  &
  *{\circ}<3pt> \ar@{-}[r]_<{5}   &*{\circ}<3pt> \ar@{}[]_<{6}
\\ & & *{\bullet}<3pt> \ar@{-}[u]^<{2}  }}
\hspace{2cm}
\raisebox{1em}{\xymatrix@R=3ex{ *{x_2}<3pt> \ar@{-}[r]_<{}  &
*{x_1}<3pt> \ar@{-}[r]_<{}  &*{u}<3pt>
\ar@{-}[r]_<{}  &
  *{y_1}<3pt> \ar@{-}[r]_<{}   &*{y_2}<3pt> \ar@{}[]_<{}
\\ & & *{z_1}<3pt> \ar@{-}[u]^<{}  }}
\]

\subsubsection{ $V(\omega_6)$}

There are  3 graded components.
\begin{center}

\begin{tabular}{|c||l|ll|c|}
\hline
{\bf 0} & $F_1$ &&(1,0,0,0,0,0)&1 \\ \hline
{\bf 1} & $\bigwedge^4F_1$ &&(1,1,1,1,0,0) &1 \\ \hline
{\bf 2} & $S_{2,1^5}F_1$ &&(2,1,1,1,1,1) &1 \\ \hline
\end{tabular}

\end{center}

\subsubsection{ $V(\omega_1)$}

There are 3 graded components.
\begin{center}

\begin{tabular}{|c||l|ll|c|}
\hline
{\bf 0} & $F_1^*$ &&(0,0,0,0,0,-1)&1 \\ \hline
{\bf 1} & $\bigwedge^2F_1$ &&(1,1,0,0,0,0) &1 \\ \hline
{\bf 2} & $\bigwedge^5F_1$ &&(1,1,1,1,1,0) &1 \\ \hline
\end{tabular}

\end{center}

\subsubsection{ $V(\omega_2)$}

There are 5 graded components.

\begin{center}

\begin{tabular}{|c||l|ll|c||l|ll|c|}
\hline
{\bf 0} & $\mathbb C$ &&(0,0,0,0,0,0)&1 &&&& \\ \hline
{\bf 1}&$\bigwedge^3 F_1$&&(1,1,1,0,0,0)&1 & & && \\ \hline
{\bf 2} &$S_{2,1^4}F_1$&&(2,1,1,1,1,0)&1&$\bigwedge^6 F_1$& & (1,1,1,1,1,1)&1 \\ \hline
{\bf 3} &$S_{2^3,1^3}F_1 $&&(2,2,2,1,1,1)&1 &&&& \\ \hline
{\bf 4} &$S_{2^6}F_1 $&&(2,2,2,2,2,2)&1&&&& \\ \hline
\end{tabular}

\end{center}


\section{The type $E_7$}
\label{section:e7}

\subsection{ $E_7$ graded by $\alpha_5$} \label{E7A5} This format is $(1,5,7,3)$.

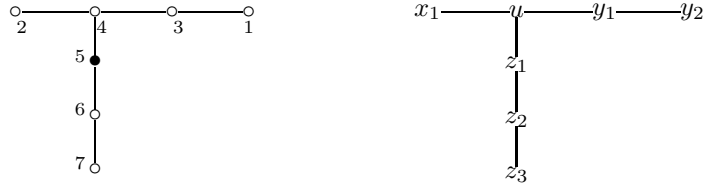
\[
\raisebox{1em}{\xymatrix@R=3ex{ 
*{\circ}<3pt> \ar@{-}[r]_<{2}  &*{\circ}<3pt>
\ar@{-}[r]_<{4}  &
  *{\circ}<3pt> \ar@{-}[r]_<{3}   &*{\circ}<3pt> \ar@{}[]_<{1}
\\ & *{\bullet}<3pt> \ar@{-}[u]^<{5} \\ & *{\circ}<3pt> \ar@{-}[u]^<{6} \\ & *{\circ}<3pt> \ar@{-}[u]^<{7} }}
\hspace{2cm}
\raisebox{1em}{\xymatrix@R=3ex{ 
*{x_1}<3pt> \ar@{-}[r]_<{}  &*{u}<3pt>
\ar@{-}[r]_<{}  &
  *{y_1}<3pt> \ar@{-}[r]_<{}   &*{y_2}<3pt> \ar@{}[]_<{}
\\ & *{z_1}<3pt> \ar@{-}[u]^<{} \\ & *{z_2}<3pt> \ar@{-}[u]^<{} \\ & *{z_3}<3pt> \ar@{-}[u]^<{} }}
\]

\subsubsection{ $V(\omega_7)$} \label{E7V7}
This representation is of dimension 56 and has 6 graded components.
\begin{center}

\begin{tabular}{|c||l|ll|c|}
\hline
{\bf 0} & $F_3$ &(0,0,-1)&(0,0,0,0,0)&1 \\ \hline
{\bf 1} & $\bigwedge^2F_1$ &(0,0,0)&(1,1,0,0,0) &1 \\ \hline
{\bf 2} & $F_3^*\otimes\bigwedge^4F_1$ &(1,0,0)&(1,1,1,1,0) &1 \\ \hline
{\bf 3} & $\bigwedge^2F_3^*\otimes S_{2,1^4} F_1$ &(1,1,0)&(2,1,1,1,1) &1  \\ \hline
{\bf 4} & $\bigwedge^3 F_3^*\otimes S_{2^3,1^2}F_1$ &(1,1,1)&(2,2,2,1,1) &1 \\ \hline
{\bf 5} & $S_{2,1,1}F_3^*\otimes S_{2^5}F_1$ &(2,1,1)&(2,2,2,2,2) &1 \\ \hline
\end{tabular}

\end{center}

\subsubsection{ $V(\omega_1)$}

This representation is of dimension 133 and has 7 graded components.
\begin{center}

\begin{tabular}{|c||l|ll|c||l|ll|c|}
\hline
{\bf 0} & $F_1^*$ &(0,0,0)&(0,0,0,0,-1)&1 &&&&  \\ \hline
{\bf 1} & $F_3^*\otimes F_1$ &(1,0,0)&(1,0,0,0,0) &1 &&&& \\ \hline
{\bf 2} & $\bigwedge^2 F_3^*\otimes\bigwedge^3F_1$ &(1,1,0)&(1,1,1,0,0) &1 &&&& \\ \hline
{\bf 3} & $\bigwedge^3 F_3^*\otimes  S_{2,1^3}F_1 $ &(1,1,1)&(2,1,1,1,0) &1  & $\bigwedge^3 F_3^*\otimes \bigwedge^5 F_1$ &(1,1,1)&(1,1,1,1,1) &1  \\ \hline & $S_{2,1}F_3^*\otimes\bigwedge^5 F_1$ &(2,1,0)&(1,1,1,1,1) &1  &&&& \\ \hline
{\bf 4} & $S_{2,1,1}F_3^*\otimes S_{2^2,1^3}F_1$ &(2,1,1)&(2,2,1,1,1) &1 &&&& \\ \hline
{\bf 5} & $S_{2,2,1}F_3^*\otimes S_{2^4,1}F_1$ &(2,2,1)&(2,2,2,2,1) &1&&&&  \\ \hline
{\bf 6} & $S_{2,2,2}F_3^*\otimes S_{3,2^4}F_1$ &(2,2,2)&(3,2,2,2,2) &1 &&&& \\ \hline
\end{tabular}

\end{center}

\pagebreak

\subsubsection{ $V(\omega_2)$}
This representation is of dimension 912 and has 10 graded components.
\begin{center}

\begin{tabular}{|c||l|ll|c||l|ll|c|}
\hline
{\bf 0} & $F_1$ &(0,0,0)&(1,0,0,0,0)&1 &&&& \\ \hline
{\bf 1} & $F_3^*\otimes \bigwedge^3 F_1$ &(1,0,0)&(1,1,1,0,0) &1 &&&& \\ \hline
{\bf 2} & $S_2F_3^*\otimes\bigwedge^5 F_1$ &(2,0,0)&(1,1,1,1,1) &1  & $\bigwedge^2 F_3^*\otimes \bigwedge^5 F_1$ &(1,1,0)&(1,1,1,1,1) &1  \\ \hline
 & $\bigwedge^2 F_3^*\otimes S_{2,1^3}F_1$ &(1,1,0)&(2,1,1,1,0) &1 &&&& \\ \hline
{\bf 3} & $S_{2,1}F_3^*\otimes S_{2^2,1^3}F_1$ &(2,1,0)&(2,2,1,1,1) &1   & $\bigwedge^3 F_3^*\otimes S_{3,1^4}F_1$ &(1,1,1)&(3,1,1,1,1) &1   \\ \hline
 & $\bigwedge^3 F_3^*\otimes S_{2^2,1^3}F_1$ &(1,1,1)&(2,2,1,1,1) &1   & $\bigwedge^3 F_3^*\otimes S_{2^3,1}F_1$ &(1,1,1)&(2,2,2,1,0) &1  \\ \hline

{\bf 4} & $S_{2,2}F_3^*\otimes S_{2^4,1}F_1$ &(2,2,0)&(2,2,2,2,1) &1 & $S_{2,1,1}F_3^*\otimes S_{2^4,1}F_1$ &(2,1,1)&(2,2,2,2,1) &2  \\ \hline
 & $S_{2,1,1}F_3^*\otimes S_{3,2,2,1,1}F_1$ &(2,1,1)&(3,2,2,1,1) &1 &&&& \\ \hline
{\bf 5} & $S_{2,2,1}F_3^*\otimes S_{3,2^4}F_1$ &(2,2,1)&(3,2,2,2,2) &2  & $S_{2,2,1}F_3^*\otimes S_{3,3,2,2,1}F_1$ &(2,2,1)&(3,3,2,2,1) &1 \\ \hline
 & $S_{3,1,1}F_3^*\otimes S_{3,2^4}F_1$ &(3,1,1)&(3,2,2,2,2) &1 &&&& \\ \hline
{\bf 6} & $S_{2,2,2}F_3^*\otimes S_{4,3,2^3}F_1$ &(2,2,2)&(4,3,2,2,2) &1  & $S_{2,2,2}F_3^*\otimes S_{3^4,1}F_1$ &(2,2,2)&(3,3,3,3,1) &1  \\ \hline
 & $S_{2,2,2}F_3^*\otimes S_{3^3,2^2}F_1$ &(2,2,2)&(3,3,3,2,2) &1 & $S_{3,2,1}F_3^*\otimes S_{3^3,2^2}F_1$ &(3,2,1)&(3,3,3,2,2) &1  \\ \hline
{\bf 7} & $S_{3,2,2}F_3^*\otimes S_{4,3^3,2}F_1$ &(3,2,2)&(4,3,3,3,2) &1  & $S_{3,2,2}F_3^*\otimes S_{3^5}F_1$ &(3,2,2)&(3,3,3,3,3) &1 \\ \hline
 & $S_{3,3,1}F_3^*\otimes S_{3^5}F_1$ &(3,3,1)&(3,3,3,3,3) &1 &&&& \\ \hline
{\bf 8} & $S_{3,3,2}F_3^*\otimes S_{4^2,3^3}F_1$ &(3,3,2)&(4,4,3,3,3) &1 &&&& \\ \hline
{\bf 9} & $S_{3,3,3}F_3^*\otimes S_{4^4,3}F_1$ &(3,3,3)&(4,4,4,4,3) &1 &&&& \\ \hline

\end{tabular}
\end{center}


\subsection{ $E_7$ graded by $\alpha_3$} This format is $(1,6,7,2)$.

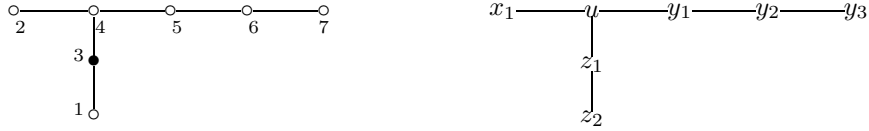
\[
\raisebox{1em}{\xymatrix@R=3ex{ 
*{\circ}<3pt> \ar@{-}[r]_<{2}  &*{\circ}<3pt>
\ar@{-}[r]_<{4}  &
  *{\circ}<3pt> \ar@{-}[r]_<{5} & *{\circ}<3pt>
\ar@{-}[r]_<{6}  &*{\circ}<3pt> \ar@{}[]_<{7}
\\ & *{\bullet}<3pt> \ar@{-}[u]^<{3} \\ & *{\circ}<3pt> \ar@{-}[u]^<{1} }}
\hspace{2cm}
\raisebox{1em}{\xymatrix@R=3ex{ 
*{x_1}<3pt> \ar@{-}[r]_<{}  &*{u}<3pt>
\ar@{-}[r]_<{}  &
  *{y_1}<3pt> \ar@{-}[r]_<{} & *{y_2}<3pt>
\ar@{-}[r]_<{}  &*{y_3}<3pt> \ar@{}[]_<{}
\\ & *{z_1}<3pt> \ar@{-}[u]^<{} \\ & *{z_2}<3pt> \ar@{-}[u]^<{} }}
\]

\subsubsection{ $V(\omega_7)$}
There are  5 graded components.
\begin{center}

\begin{tabular}{|c||l|ll|c|}
\hline
{\bf 0} & $F_1^*$ &(0,0)&(0,0,0,0,0,-1)&1 \\ \hline
{\bf 1} & $F_3^*\otimes F_1$ &(1,0)&(1,0,0,0,0,0) &1 \\ \hline
{\bf 2} & $\bigwedge^2 F_3^*\otimes\bigwedge^3F_1$ &(1,1)&(1,1,1,0,0,0) &1 \\ \hline
{\bf 3} & $S_{2,1}F_3^*\otimes \bigwedge^5F_1$ &(2,1)&(1,1,1,1,1,0) &1  \\ \hline
{\bf 4} & $S_{2,2}F_3^*\otimes S_{2,1^5}F_1$ &(2,2)&(2,1,1,1,1,1) &1 \\ \hline
\end{tabular}

\end{center}

\subsubsection{ $V(\omega_1)$}
There are  7 graded components.
\begin{center}

\begin{tabular}{|c||l|ll|c||l|ll|c|}
\hline
{\bf 0} & $F_3$ &(0,-1)&(0,0,0,0,0,0)&1 &&&& \\ \hline
{\bf 1} & $\bigwedge^2 F_1$&(0,0)&(1,1,0,0,0,0)&1 & & &&\\ \hline
{\bf 2} & $F_3^*\otimes \bigwedge^4 F_1$&(1,0)&(1,1,1,1,0,0)&1& &  & &\\ \hline
{\bf 3} & $\bigwedge^2 F_3^*\otimes S_{2,1^4} F_1$&(1,1)&(2,1,1,1,1,0)&1 &$\bigwedge^2 F_3^*\otimes \bigwedge^6 F_1$ & (1,1) & (1,1,1,1,1,1) & 1  \\ \hline
&$S_2 F_3^*\otimes\bigwedge^6F_1$ &(2,0) & (1,1,1,1,1,1) & 1 &&&&\\ \hline
{\bf 4} &$S_{2,1}F_3^*\otimes S_{2,2,1^4}F_1$ &(2,1)&(2,2,1,1,1,1)&1& &&& \\ \hline
{\bf 5} &$S_{2,2}F_3^*\otimes S_{2^4,1,1}F_1$&(2,2)&(2,2,2,2,1,1)&1& & &&  \\ \hline
{\bf 6} &$S_{3,2}F_3^*\otimes S_{2^6}F_1$&(3,2)&(2,2,2,2,2,2)&1& & & &\\ \hline
\end{tabular}

\end{center}

\pagebreak

\subsubsection{ $V(\omega_2)$}

We have  9 graded components.

\begin{center}

\begin{tabular}{|c||l|ll|c||l|ll|c|}
\hline
{\bf 0} & $F_1$ &(0,0)&(1,0,0,0,0,0)&1 &&&& \\ \hline
{\bf 1}&$F_3^*\otimes\bigwedge^3 F_1$&(1,0)&(1,1,1,0,0,0)&1 & & && \\ \hline
{\bf 2} &$\bigwedge^2 F_3^*\otimes S_{2,1,1,1}F_1$&(1,1)&(2,1,1,1,0,0)&1&$\bigwedge^2 F_3^*\otimes\bigwedge^5 F_1$&(1,1) & (1,1,1,1,1,0)&1 \\ \hline
&$S_2F_3^*\otimes\bigwedge^5F_1$&(2,0) & (1,1,1,1,1,0)&1& & & &\\ \hline
{\bf 3} &$S_{2,1}F_3^*\otimes S_{2,2,1,1,1}F_1 $&(2,1)&(2,2,1,1,1,0)&1 &$S_{2,1}F_3^*\otimes S_{2,1^5}F_1$& (2,1) & (2,1,1,1,1,1) & 2 \\ \hline
{\bf 4} &$S_{2,2}F_3^*\otimes S_{2^3,1^3}F_1 $&(2,2)&(2,2,2,1,1,1)&1& $S_{2,2}F_3^*\otimes S_{3,2,1^4}F_1 $&(2,2)& (3,2,1,1,1,1)& 1 \\ \hline
&$S_{2,2}F_3^*\otimes S_{2,2,2,2,1}F_1$&(2,2) & (2,2,2,2,1,0)& 1 &$S_{3,1}F^*_3\otimes S_{2^3,1^3}F_1$ & (3,1)& (2,2,2,1,1,1)& 1 \\ \hline
{\bf 5} &$S_{3,2}F_3^*\otimes S_{3,2^3,1^2}F_1$&(3,2)&(3,2,2,2,1,1)&1& $S_{3,2}F_3^*\otimes S_{2^5,1}F_1$& (3,2)& (2,2,2,2,2,1)& 2  \\ \hline
{\bf 6} &$S_{3,3}F_3^*\otimes S_{3^2,2^3,1} F_1$&(3,3)&(3,3,2,2,2,1)&1& $S_{3,3}F_3^*\otimes S_{3,2^5}F_1$&(3,3)& (3,2,2,2,2,2)& 1 \\ \hline
&$S_{4,2}F_3^*\otimes S_{3,2^5}F_1$&(4,2) & (3,2,2,2,2,2)& 1& & & &\\ \hline
{\bf 7}&$S_{4,3} F_3^*\otimes S_{3^3,2^3} F_1$&(4,3)&(3,3,3,2,2,2)&1 & & && \\ \hline
{\bf 8}&$S_{4,4}F_3^*\otimes S_{3^5,2}F_1$&(4,4)&(3,3,3,3,3,2)&1 & & && \\ \hline
\end{tabular}

\end{center}


\subsection{ $E_7$ graded by $\alpha_2$} This format is $(2,7,6,1)$.

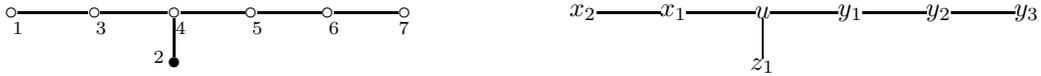
\[
\raisebox{1em}{\xymatrix@R=3ex{ *{\circ}<3pt> \ar@{-}[r]_<{1}  &
*{\circ}<3pt> \ar@{-}[r]_<{3}  &*{\circ}<3pt>
\ar@{-}[r]_<{4}  &
  *{\circ}<3pt> \ar@{-}[r]_<{5} & *{\circ}<3pt>
\ar@{-}[r]_<{6}  &*{\circ}<3pt> \ar@{}[]_<{7}
\\ & & *{\bullet}<3pt> \ar@{-}[u]^<{2}  }}
\hspace{2cm}
\raisebox{1em}{\xymatrix@R=3ex{ *{x_2}<3pt> \ar@{-}[r]_<{}  &
*{x_1}<3pt> \ar@{-}[r]_<{}  &*{u}<3pt>
\ar@{-}[r]_<{}  &
  *{y_1}<3pt> \ar@{-}[r]_<{} & *{y_2}<3pt>
\ar@{-}[r]_<{}  &*{y_3}<3pt> \ar@{}[]_<{}
\\ & & *{z_1}<3pt> \ar@{-}[u]^<{}  }}
\]

\subsubsection{ $V(\omega_7)$}

There are 4 graded components.
\begin{center}

\begin{tabular}{|c||l|ll|c|}
\hline
{\bf 0} & $F_1^*$ &&(0,0,0,0,0,0,-1)&1 \\ \hline
{\bf 1} & $\bigwedge^2F_1$ &&(1,1,0,0,0,0,0) &1 \\ \hline
{\bf 2} & $\bigwedge^5F_1$ &&(1,1,1,1,1,0,0) &1 \\ \hline
{\bf 3} & $S_{2,1^6}F_1$ &&(2,1,1,1,1,1,1) &1 \\ \hline
\end{tabular}

\end{center}

\subsubsection{ $V(\omega_1)$}

There are  5 graded components.
\begin{center}

\begin{tabular}{|c||l|ll|c||l|ll|c|}
\hline
{\bf 0} & $F_1$ &&(1,0,0,0,0,0,0)&1&&&& \\ \hline
{\bf 1} & $\bigwedge^4F_1$ &&(1,1,1,1,0,0,0) &1 &&&&\\ \hline
{\bf 2} & $S_{2,1^5}F_1$ &&(2,1,1,1,1,1,0)
 &1 & $\bigwedge^7F_1$ &&(1,1,1,1,1,1,1)
 &1 \\ \hline
{\bf 3} & $S_{2^3,1^4}F_1$ &&(2,2,2,1,1,1,1)
 &1&&&& \\ \hline
{\bf 4} & $S_{2^6,1}F_1$ &&(2,2,2,2,2,2,1)
 &1 &&&&\\ \hline
\end{tabular}

\end{center}

\subsubsection{ $V(\omega_2)$}
There are  8 graded components.

\begin{center}

\begin{tabular}{|c||l|ll|c||l|ll|c|}
\hline
{\bf 0} & $\mathbb C$  &&  (0,0,0,0,0,0,0) & 1 &&&& \\ \hline
{\bf 1} & $\bigwedge^3, F_1$  &&  (1,1,1,0,0,0,0) & 1 & & && \\ \hline
{\bf 2} & $S_{2,1^4}F_1$  &&  (2,1,1,1,1,0,0) & 1
& $\bigwedge^6 F_1$ &&  (1,1,1,1,1,1,0) & 1 \\ \hline
{\bf 3} & $S_{3,1^6}F_1$  &&  (3,1,1,1,1,1,1) & 1
& $S_{2^3,1,^3,}F_1$  &&  (2,2,2,1,1,1,0) & 1 \\ \hline
& $S_{2^2,1^5}F_1$ &&  (2,2,1,1,1,1,1) & 1 &&&& \\ \hline
{\bf 4}  & $S_{3,2^3,1^3}F_1$ &&  (3,2,2,2,1,1,1) & 1
& $S_{2^6}F_1$ &&  (2,2,2,2,2,2,0) & 1 \\ \hline
& $S_{2^5,1^2}F_1$ &&  (2,2,2,2,2,1,1) & 1 &&&& \\ \hline
{\bf 5} & $S_{3^2,2^4,1}F_1$  &&  (3,3,2,2,2,2,1) & 1
& $S_{3,2^6}F_1$ &&  (3,2,2,2,2,2,2) & 1 \\ \hline
{\bf 6} & $S_{3^4,2^3,}F_1$ &&  (3,3,3,3,2,2,2) & 1 &&&& \\ \hline
{\bf 7} & $S_{3^7}F_1$ &&  (3,3,3,3,3,3,3) & 1 &&&& \\ \hline
\end{tabular}

\end{center}



\section{The type $E_8$}
\label{section:e8}

{

\subsection{ $E_8$ graded by $\alpha_5$} This format is $(1,5,8,4)$.

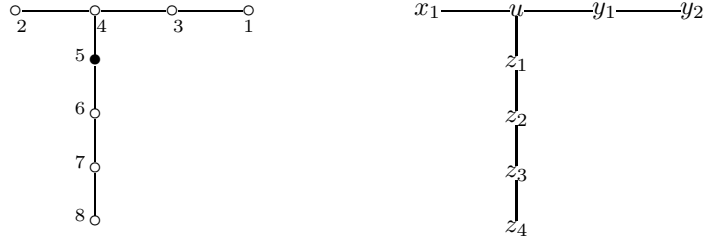
\[
\raisebox{1em}{\xymatrix@R=3ex{ 
*{\circ}<3pt> \ar@{-}[r]_<{2}  &*{\circ}<3pt>
\ar@{-}[r]_<{4}  &
  *{\circ}<3pt> \ar@{-}[r]_<{3}   &*{\circ}<3pt> \ar@{}[]_<{1}
\\ & *{\bullet}<3pt> \ar@{-}[u]^<{5} \\ & *{\circ}<3pt> \ar@{-}[u]^<{6} \\ & *{\circ}<3pt> \ar@{-}[u]^<{7} \\ & *{\circ}<3pt> \ar@{-}[u]^<{8}}}
\hspace{2cm}
\raisebox{1em}{\xymatrix@R=3ex{ 
*{x_1}<3pt> \ar@{-}[r]_<{}  &*{u}<3pt>
\ar@{-}[r]_<{}  &
  *{y_1}<3pt> \ar@{-}[r]_<{}   &*{y_2}<3pt> \ar@{}[]_<{}
\\ & *{z_1}<3pt> \ar@{-}[u]^<{} \\ & *{z_2}<3pt> \ar@{-}[u]^<{} \\ & *{z_3}<3pt> \ar@{-}[u]^<{} \\ & *{z_4}<3pt> \ar@{-}[u]^<{}}}
\]

\subsubsection{ $V(\omega_8)$} 
This representation is of dimension 248 and has 11 graded components.

\begin{center}
\begin{tabular}{|c||l|ll|c||l|ll|c|}
\hline
{\bf 0} & $F_3$ &(0,0,0,-1)&(0,0,0,0,0)&1 &&&&  \\ \hline
{\bf 1} & $\bigwedge^2F_1$ &(0,0,0,0)&(1,1,0,0,0) &1 &&&& \\ \hline
{\bf 2} & $F_3^*\otimes\bigwedge^4F_1$ &(1,0,0,0)&(1,1,1,1,0) &1 &&&& \\ \hline
{\bf 3} & $\bigwedge^2F_3^*\otimes S_{2,1^4} F_1$ &(1,1,0,0)&(2,1,1,1,1) &1 &&&& \\ \hline
{\bf 4} & $\bigwedge^3 F_3^*\otimes S_{2^3,1^2}F_1$ &(1,1,1,0)&(2,2,2,1,1) &1 &&&& \\ \hline
{\bf 5} & $\bigwedge^4 F_3^*\otimes S_{3,2^3,1}F_1$ & (1,1,1,1)&(3,2,2,2,1)&1 & $\bigwedge^4 F_3^*\otimes S_{2^5}F_1$ &(1,1,1,1)&(2,2,2,2,2) &1  \\ \hline & $S_{2,1,1}F_3^*\otimes S_{2^5}F_1$ &(2,1,1,0)&(2,2,2,2,2) &1  &&&&  \\ \hline
{\bf 6} & $S_{2,1^3}F_3^*\otimes S_{3^2,2^3}F_1$ &(2,1,1,1)&(3,3,2,2,2) &1 &&&& \\ \hline
{\bf 7} & $S_{2^2,1^2}F^*_3\otimes S_{3^4,2}F_1$ &(2,2,1,1)&(3,3,3,3,2) &1 &&&& \\ \hline
{\bf 8} & $S_{2^3,1}F_3^*\otimes S_{4,3^4}F_1$ &(2,2,2,1)&(4,3,3,3,3) &1 &&&& \\ \hline
{\bf 9} & $S_{2^4}F^*_3\otimes S_{4^3,3^2}F_1$ &(2,2,2,2)&(4,4,4,3,3) &1 &&&& \\ \hline
{\bf 10} & $S_{3,2^3}F^*_3\otimes S_{4^5}F_1$ &(3,2,2,2)&(4,4,4,4,4) &1 &&&& \\ \hline
\end{tabular}

\end{center}
%

\subsubsection{ $V(\omega_1)$}

This representation is of dimension 3875 and has 17 graded components.

\begin{center}
\begin{tabular}{|c||l|ll|c||l|ll|c|}
\hline
{\bf 0} & $F_1^*$ &(0,0,0,0)  & (0,0,0,0,-1)&1 &&&& \\ \hline
{\bf 1} & $F_3^*\otimes F_1$ &(1,0,0,0) & (1,0,0,0,0) &1 &&&& \\ \hline
{\bf 2} & $\bigwedge^2 F_3^*\otimes\bigwedge^3F_1$ &(1,1,0,0) & (1,1,1,0,0) &1  &&&& \\ \hline

{\bf 3} & $\bigwedge^3 F_3^*\otimes S_{2,1^3}F_1$ &(1,1,1,0) & (2,1,1,1,0) &1   & $\bigwedge^3 F_3^*\otimes \bigwedge^5 F_1$ &(1,1,1,0) & (1,1,1,1,1) &1   \\ \hline
 & $S_{2,1}F_3^*\otimes\bigwedge^5 F_1$ &(2,1,0,0) & (1,1,1,1,1) &1 &&&& \\ \hline

{\bf 4} & $\bigwedge^4 F_3^*\otimes S_{2^2,1^3}F_1$ &(1,1,1,1) & (2,2,1,1,1) &1 & $\bigwedge^4 F_3^*\otimes S_{2^3,1}F_1$ &(1,1,1,1) & (2,2,2,1,0) &1  \\ \hline
 & $\bigwedge^4 F_3^*\otimes S_{3,1^4}F_1$ &(1,1,1,1)  & (3,1,1,1,1) &1 &$S_{2,1,1}F_3^*\otimes S_{2^2,1^3}F_1$ &(2,1,1,0)  & (2,2,1,1,1) &1  \\ \hline

{\bf 5} & $S_{2,1,1,1}F_3^*\otimes S_{2^4,1}F_1$ &(2,1,1,1)  & (2,2,2,2,1) &2  & $S_{2,1,1,1}F_3^*\otimes S_{3,2^2,1^2}F_1$ & (2,1,1,1)  & (3,2,2,1,1) &1 \\ \hline
 & $S_{2,2,1}F_3^*\otimes S_{2^4,1}F_1$ &(2,2,1,0)  & (2,2,2,2,1) &1 &&&& \\ \hline

{\bf 6} & $S_{2^2,1^2}F_3^*\otimes S_{3,2^4}F_1$ &(2,2,1,1) & (3,2,2,2,2) &2  & $S_{2^2,1^2}F_3^*\otimes S_{3^2,2^2,1}F_1$ &(2,2,1,1) & (3,3,2,2,1) &1  \\ \hline
 & $S_{3,1^3}F_3^*\otimes S_{3,2^4}F_1$ &(3,1,1,1)  & (3,2,2,2,2) &1 & $S_{2,2,2}F_3^*\otimes S_{3,2^4}F_1$ &(2,2,2,0)  & (3,2,2,2,2) &1  \\ \hline

{\bf 7} & $S_{2^3,1}F^*_3\otimes S_{3^3,2^2}F_1$ &(2,2,2,1) & (3,3,3,2,2) &2  & $S_{2^3,1}F^*_3\otimes S_{4,3,2^3}F_1$ &(2,2,2,1) & (4,3,2,2,2) &1 \\ \hline
 & $S_{2^3,1}F^*_3\otimes S_{3^4,1}F_1$ &(2,2,2,1)  & (3,3,3,3,1) &1 & $S_{3,2,1,1}F^*_3\otimes S_{3^3,2^2}F_1$ &(3,2,1,1)  & (3,3,3,2,2) &1 \\ \hline

{\bf 8} & $S_{2^4}F^*_3\otimes S_{4,3^3,2}F_1$ &(2,2,2,2) & (4,3,3,3,2) &2  & $S_{2^4}F^*_3\otimes S_{3^5}F_1$ &(2,2,2,2) & (3,3,3,3,3) &2 \\ \hline
 & $S_{2^4}F^*_3\otimes S_{4^2,3,2^2}F_1$ &(2,2,2,2)  & (4,4,3,2,2) &1 & $S_{3,2,2,1}F^*_3\otimes S_{3^5}F_1$ &(3,2,2,1)  & (3,3,3,3,3) &2 \\ \hline
& $S_{3,2,2,1}F^*_3 \otimes S_{4,3^3,2}F_1$ &(3,2,2,1) & (4,3,3,3,2) &1 & $S_{3^2,1^2}F^*_3\otimes S_{3^5}F_1$ &(3,3,1,1)  & (3,3,3,3,3) &1 \\ \hline

\end{tabular}
\end{center}
\hfill (The table continues on the next page.)

\begin{center}
\begin{tabular}{|c||l|ll|c||l|ll|c|}
\hline

{\bf 9} & $S_{3,2^3}F_3^*\otimes S_{4^2,3^3}F_1$ &(3,2,2,2) & (4,4,3,3,3)&2  & $S_{3,2^3}F_3^*\otimes S_{4^3,3,2}F_1$ &(3,2,2,2)  & (4,4,4,3,2) &1 \\ \hline
 & $S_{3,2^3}F_3^*\otimes S_{5,3^4}F_1$ &(3,2,2,2)  & (5,3,3,3,3) &1 & $S_{3^2,2,1}F_3^*\otimes S_{4^2,3^3}F_1$ &(3,3,2,1)  & (4,4,3,3,3) &1 \\ \hline

{\bf 10} & $S_{3^2,2^2}F_3^*\otimes S_{4^4,3}F_1$ &(3,3,2,2)  & (4,4,4,4,3) &2  & $S_{3^2,2^2} F_3^*\otimes S_{5,4^2,3^2}F_1$ &(3,3,2,2)  & (5,4,4,3,3)&1   \\ \hline
 & $S_{4,2^3} F_3^*\otimes S_{4^4,3}F_1$ &(4,2,2,2) & (4,4,4,4,3) &1   & $S_{3^3,1}\otimes S_{4^4,3}F_1$ &(3,3,3,1) & (4,4,4,4,3) &1  \\ \hline

{\bf 11} & $S_{3^3,2}F_3^*\otimes S_{5, 4^4}F_1$ &(3,3,3,2)  & (5,4,4,4,4)&2 & $S_{3^3,2}F_3^*\otimes S_{5^2,4^2,3}F_1$ &(3,3,3,2)  & (5,5,4,4,3) &1  \\ \hline
 & $S_{4,3,2^2}F_3^*\otimes S_{5,4^4}F_1$ &(4,3,2,2)  & (5,4,4,4,4)&1 &&&& \\ \hline

{\bf 12} & $S_{3^4}F_3^*\otimes S_{5^3,4^2}F_1$ &(3,3,3,3) & (5,5,5,4,4) &1  & $S_{3^4}F_3^*\otimes S_{6,5,4^3}F_1$ &(3,3,3,3) & (6,5,4,4,4) &1 \\ \hline
 & $S_{3^4}F_3^*\otimes S_{5^4,3}F_1$ &(3,3,3,3)  & (5,5,5,5,3)&1 & $S_{4,3^2,2}F_3^*\otimes S_{5^3,4^2}F_1$ &(4,3,3,2) & (5,5,5,4,4)&1 \\ \hline

{\bf 13} & $S_{4,3^3}F_3^*\otimes S_{5^5}F_1$ &(4,3,3,3)  & (5,5,5,5,5)&1  & $S_{4,3^3}F_3^*\otimes S_{6,5^3,4}F_1$ &(4,3,3,3)  & (6,5,5,5,4) &1  \\ \hline
 & $S_{4^2,3, 2}F_3^*\otimes S_{5^5}F_1$ &(4,4,3,2)  & (5,5,5,5,5) &1 &&&&  \\ \hline

{\bf 14} & $S_{4^2,3^2}F_3^*\otimes S_{6^2,5^3}F_1$ &(4,4,3,3)  & (6,6,5,5,5)&1 &&&& \\ \hline

{\bf 15} & $S_{4^3,3}F_3^*\otimes S_{6^4,5}F_1$ &(4,4,4,3)  & (6,6,6,6,5)&1 &&&& \\ \hline

{\bf 16} & $S_{4^4}F_3^*\otimes S_{7, 6^4}F_1$ &(4,4,4,4)  & (7,6,6,6,6) &1 &&&& \\ \hline
\end{tabular}
\end{center}


\subsubsection{ $V(\omega_2)$}

This representation is of dimension 147250 and has 25 graded components. We exhibit the first 13, as the others can be determined
by duality. (The representation is graded self-dual.)

\begin{center}
\begin{tabular}{|c||l|ll|c||l|ll|c|}
\hline
{\bf 0} & $F_1$ &(0,0,0,0)  & (1,0,0,0,0) & 1
 &&&& \\ \hline

{\bf 1} & $F_3^*\otimes \bigwedge^3 F_1$ &(1,0,0,0) & (1,1,1,0,0) & 1 &&&& \\ \hline

{\bf 2}  & $S_2F_3^*\otimes\bigwedge^5 F_1$ &(2,0,0,0)  & (1,1,1,1,1) & 1   & $\bigwedge^2 F_3^*\otimes \bigwedge^5 F_1$ &(1,1,0,0)  & (1,1,1,1,1) & 1   \\ \hline
& $\bigwedge^2 F_3^*\otimes S_{2,1^3}F_1$ &(1,1,0,0) & (2,1,1,1,0) & 1  &&&& \\ \hline

{\bf 3} & $S_{2,1}F_3^*\otimes S_{2^2,1^3}F_1$ &(2,1,0,0) & (2,2,1,1,1) & 1   & $\bigwedge^3 F_3^*\otimes S_{3,1^4}F_1$ &(1,1,1,0) & (3,1,1,1,1) & 1  \\ \hline
& $\bigwedge^3 F_3^*\otimes S_{2^3,1}F_1$ &(1,1,1,0) & (2,2,2,1,0) & 1   & $\bigwedge^3 F_3^*\otimes S_{2^2,1^3}F_1$ &(1,1,1,0) & (2,2,1,1,1) & 1   \\ \hline

{\bf 4} & $S_{2,2}F_3^*\otimes S_{2^4,1}F_1$ &(2,2,0,0) & (2,2,2,2,1) & 1& $S_{2,1,1}F_3^*\otimes S_{2^4,1}F_1$ &(2,1,1,0) & (2,2,2,2,1) & 2 \\ \hline
 & $S_{2,1,1}F_3^*\otimes S_{3,2,2,1,1}F_1$ &(2,1,1,0) & (3,2,2,1,1) & 1 & $\bigwedge^4F^*_3\otimes S_{3,2^2,1^2}F_1$ &(1,1,1,1) & (3,2,2,1,1) & 2 \\ \hline
& $\bigwedge^4F^*_3\otimes S_{2^4,1}F_1$ &(1,1,1,1) & (2,2,2,2,1) & 2 & $\bigwedge^4F^*_3\otimes S_{3,2^3}F_1$ &(1,1,1,1) & (3,2,2,2,0) & 1  \\ \hline

{\bf 5} & $S_{3,1,1}F_3^*\otimes S_{3,2^4}F_1$ &(3,1,1,0) & (3,2,2,2,2) & 1  & $S_{2,2,1}F_3^*\otimes S_{3,2^4}F_1$ & (2,2,1,0) & (3,2,2,2,2) & 2 \\ \hline
&$S_{2,2,1}F_3^*\otimes S_{3,3,2,2,1}F_1$ &(2,2,1,0) & (3,3,2,2,1) & 1  & $S_{2,1^3}F^*_3\otimes S_{3,2^4}F_1$ & (2,1,1,1) & (3,2,2,2,2) & 4 \\ \hline
&$S_{2,1^3}F^*_3\otimes S_{3^2,2^2,1}F_1$ &(2,1,1,1) & (3,3,2,2,1) & 2  & $S_{2,1^3}F^*_3\otimes S_{4,2^3,1}F_1$ & (2,1,1,1) & (4,2,2,2,1) & 1 \\ \hline
 & $S_{2,1^3}F^*_3\otimes S_{3^3,1^2}F_1$ &(2,1,1,1) & (3,3,3,1,1) & 1 &&&& \\ \hline

{\bf 6} & $S_{3,2,1}F_3^*\otimes S_{3^3,2^2}F_1$ &(3,2,1,0) & (3,3,3,2,2) & 1  & $S_{3,1^3}F^*_3\otimes S_{3^3,2^2}F_1]$ &(3,1,1,1) & (3,3,3,2,2) & 2  \\ \hline
 & $S_{3,1^3}F^*_3\otimes S_{4,3,2^3}F_1$ &(3,1,1,1) & (4,3,2,2,2) & 1 & $S_{2,2,2}F_3^*\otimes S_{3^4,1}F_1$ &(2,2,2,0) & (3,3,3,3,1) & 1  \\ \hline
 & $S_{2,2,2}F_3^*\otimes S_{4,3,2^3}F_1$ &(2,2,2,0) & (4,3,2,2,2) & 1 & $S_{2,2,2}F_3^*\otimes S_{3^3,2^2}F_1$ &(2,2,2,0) & (3,3,3,2,2) & 1  \\ \hline
  & $S_{2^2,1^2}F^*_3\otimes S_{3^3,2^2}F_1$ &(2,2,1,1) & (3,3,3,2,2) & 5 & $S_{2^2,1^2}F^*_3\otimes S_{3^4,1}F_1$ &(2,2,1,1) & (3,3,3,3,1) & 2  \\ \hline
   & $S_{2^2,1^2}F^*_3\otimes S_{4,3,2^3}F_1$ &(2,2,1,1) & (4,3,2,2,2) & 2 & $S_{2^2,1^2}F^*_3\otimes S_{4,3^2,2,1}F_1$ &(2,2,1,1) & (4,3,3,2,1) & 1  \\ \hline

{\bf 7} & $S_{4,1,1,1}F^*_3\otimes S_{3^5}F_1$ &(4,1,1,1) & (3,3,3,3,3) & 1  & $S_{3,3,1}F_3^*\otimes S_{3^5}F_1$ &(3,3,1,0) & (3,3,3,3,3) & 1 \\ \hline
 & $S_{3,2,2}F_3^*\otimes S_{3^5}F_1$ &(3,2,2,0) & (3,3,3,3,3) & 1 & $S_{3,2,2}F_3^*\otimes S_{4,3^3,2}F_1$ &(3,2,2,0) & (4,3,3,3,2) & 1 \\ \hline
 & $S_{3,2,1^2}F^*_3\otimes S_{3^5}F_1$ &(3,2,1,1) & (3,3,3,3,3) & 4 & $S_{3,2,1^2}F^*_3\otimes S_{4,3^3,2}F_1$ &(3,2,1,1) & (4,3,3,3,2) & 3 \\ \hline 
 & $S_{3,2,1^2}F^*_3\otimes S_{4^2,3,2^2}F_1$ &(3,2,1,1) & (4,4,3,2,2) & 1 & $S_{2^3,1}F^*_3\otimes S_{4,3^3,2}F_1$ &(2,2,2,1) & (4,3,3,3,2) & 6 \\ \hline 
 & $S_{2^3,1}F^*_3\otimes S_{3^5}F_1$ &(2,2,2,1) & (3,3,3,3,3) & 4 & $S_{2^3,1}F^*_3\otimes S_{4^2,3,2^2}F_1$ &(2,2,2,1) & (4,4,3,2,2) & 2 \\ \hline 
 & $S_{2^3,1}F^*_3\otimes S_{5,3^2,2^2}F_1$ &(2,2,2,1) & (5,3,3,2,2) & 1 & $S_{2^3,1}F^*_3\otimes S_{4^2,3^2,1}F_1$ &(2,2,2,1) & (4,4,3,3,1) & 1 \\ \hline

\end{tabular}
\end{center}
\hfill (The table continues on the next page.)

\begin{center}
\begin{tabular}{|c||l|ll|c||l|ll|c|}
\hline
{\bf 8} & $S_{4,2,1^2}F^*_3\otimes S_{4^2,3^3}F_1$ &(4,2,1,1) & (4,4,3,3,3) & 1  & $S_{3,3,2}F_3^*\otimes S_{4^2,3^3}F_1$ &(3,3,2,0) & (4,4,3,3,3) & 1 \\ \hline
 & $S_{3^2,1^2}F^*_3\otimes S_{4^2,3^3}F_1$ &(3,3,1,1) & (4,4,3,3,3) & 2 & $S_{3^2,1^2}F_3^*\otimes S_{4^3,3,2}F_1$ &(3,3,1,1) & (4,4,4,3,2) & 1 \\ \hline
& $S_{3,2^2,1}F^*_3\otimes S_{4^2,3^3}F_1$ &(3,2,2,1) & (4,4,3,3,3) & 6 & $S_{3,2^2,1}F^*_3\otimes S_{5,3^4}F_1$ &(3,2,2,1) & (5,3,3,3,3) & 3 \\ \hline

& $S_{3,2^2,1}F^*_3\otimes S_{4^3,3,2}F_1$ &(3,2,2,1) & (4,4,4,3,2) & 3 & $S_{3,2^2,1}F^*_3\otimes S_{5,4,3^2,2}F_1$ &(3,2,2,1) & (5,4,3,3,2) & 1 \\ \hline
& $S_{2^4}F^*_3\otimes S_{4^2,3^3}F_1$ &(2,2,2,2) & (4,4,3,3,3) & 5 & $S_{2^4}F^*_3\otimes S_{4^3,3,2}F_1$ &(2,2,2,2) & (4,4,4,3,2) & 4 \\ \hline

& $S_{2^4}F^*_3\otimes S_{5,3^4}F_1$ &(2,2,2,2) & (5,3,3,3,3) & 3 & $S_{2^4}F^*_3\otimes S_{5,4,3^2,2}F_1$ &(2,2,2,2) & (5,4,3,3,2) & 2 \\ \hline
& $S_{2^4}F^*_3\otimes S_{4^4,1}F_1$ &(2,2,2,2) & (4,4,4,4,1) & 1 & $S_{2^4}F^*_3\otimes S_{5,4^2,2^2}F_1$ &(2,2,2,2) & (5,4,4,2,2) & 1 \\ \hline

{\bf 9} & $S_{4,3,1,1}F^*_3\otimes S_{4^4,3}F_1$ &(4,3,1,1) & (4,4,4,4,3) & 1 & $
S_{4,2^2,1}F^*_3\otimes S_{4^4,3}F_1$ &(4,2,2,1) & (4,4,4,4,3) & 2  \\ \hline

& $S_{4,2^2,1}F^*_3\otimes S_{5,4^2,3^2}F_1$ &(4,2,2,1) & (5,4,4,3,3) & 1 

 & $S_{3,3,3}F_3^*\otimes S_{4^4,3}F_1$ &(3,3,3,0) & (4,4,4,4,3) & 1 \\ \hline 

 & $S_{3^2,2,1}F^*_3\otimes S_{4^4,3}F_1$ &(3,3,2,1) & (4,4,4,4,3) & 6

& $S_{3^2,2,1}F^*_3\otimes S_{5,4^2,3^2}F_1$ &(3,3,2,1) & (5,4,4,3,3) & 3 \\ \hline

& $S_{3^2,2,1}F^*_3\otimes S_{5^2,3^3}F_1$ &(3,3,2,1) & (5,5,3,3,3) & 1 
& $S_{3^2,2,1}F^*_3\otimes S_{5,4^3,2}F_1$ &(3,3,2,1) & (5,4,4,4,2) & 1 \\ \hline 

& $S_{3,2^3}F^*_3\otimes S_{4^4,3}F_1$ &(3,2,2,2) & (4,4,4,4,3) & 8

& $S_{3,2^3}F^*_3\otimes S_{5,4^2,3^2}F_1$ &(3,2,2,2) & (5,4,4,3,3) & 7 \\ \hline 

& $S_{3,2^3}F^*_3\otimes S_{5,4^3,2}F_1$ &(3,2,2,2) & (5,4,4,4,2) & 2

& $S_{3,2^3}F^*_3\otimes S_{6,4,3^3}F_1$ &(3,2,2,2) & (6,4,3,3,3) & 1 \\ \hline

& $S_{3,2^3}F^*_3\otimes S_{5^2,3^3}F_1$ &(3,2,2,2) & (5,5,3,3,3) & 1 
& $S_{3,2^3}F^*_3\otimes S_{5^2,4,3,2}F_1$ &(3,2,2,2) & (5,5,4,3,2) & 1 \\ \hline

{\bf 10} & $S_{4,3,2,1}F^*_3\otimes S_{5,4^4}F_1$ &(4,3,2,1) & (5,4,4,4,4) & 3  

& $S_{4,3,2,1}F^*_3\otimes S_{5^2,4^2,3}F_1$ &(4,3,2,1) & (5,5,4,4,3) & 1   \\ \hline

 & $S_{4,2^3}F^*_3\otimes S_{5,4^4}F_1$ &(4,2,2,2) & (5,4,4,4,4) & 5   

 & $S_{4,2^3}F^*_3\otimes S_{5^2,4^2,3}F_1$ &(4,2,2,2) & (5,5,4,4,3) & 2  \\ \hline

& $S_{4,2^3}F^*_3\otimes S_{6,4^3,3}F_1$ &(4,2,2,2) & (6,4,4,4,3) & 1  

& $S_{4,2^3}F^*_3\otimes S_{5^3,3^2}F_1$ &(4,2,2,2) & (5,5,5,3,3) & 1  \\ \hline

& $S_{3^3,1}F^*_3\otimes S_{5,4^4}F_1$ &(3,3,3,1) & (5,4,4,4,4) & 4  

& $S_{3^3,1}F^*_3\otimes S_{5^2,4^2,3}F_1$ &(3,3,3,1) & (5,5,4,4,3) & 3  \\ \hline

& $S_{3^3,1}F^*_3\otimes S_{6,4^3,3}F_1$ &(3,3,3,1) & (6,4,4,4,3) & 1

& $S_{3^2,2^2}F^*_3\otimes S_{5,4^4}F_1$ &(3,3,2,2) & (5,4,4,4,4) & 9  \\ \hline

& $S_{3^2,2^2}F^*_3\otimes S_{5^2,4^2,3}F_1$ &(3,3,2,2) & (5,5,4,4,3) & 7  

& $S_{3^2,2^2}F^*_3\otimes S_{6,4^3,3}F_1$ &(3,3,2,2) & (6,4,4,4,3) & 2  \\ \hline

& $S_{3^2,2^2}F^*_3\otimes S_{5^3,3^2}F_1$ &(3,3,2,2) & (5,5,5,3,3) & 2   

& $S_{3^2,2^2}F^*_3\otimes S_{5^3,4,2}F_1$ &(3,3,2,2) & (5,5,5,4,2) & 1 \\ \hline

& $S_{3^2,2^2}F^*_3\otimes S_{6,5,4,3^2}F_1$ &(3,3,2,2) & (6,5,4,3,3) & 1 & &&& \\ \hline

{\bf 11} & $S_{5,2^3}F^*_3\otimes S_{5^3,4^2}F_1$ &(5,2,2,2) & (5,5,5,4,4) & 1

& $S_{4^2, 2,1}F_3^*\otimes S_{5^3,4^2}F_1$ &(4,4,2,1) & (5,5,5,4,4) & 1 \\ \hline

& $S_{4,3^2,1}F^*_3\otimes S_{5^3,4^2}F_1$ &(4,3,3,1) & (5,5,5,4,4) & 2  

& $S_{4,3^2,1}F^*_3\otimes S_{5^4,3}F_1$ &(4,3,3,1) & (5,5,5,5,3) & 1\\ \hline

& $S_{4,3^2,1}F^*_3\otimes S_{6,5,4^3}F_1$ &(4,3,3,1) & (6,5,4,4,4) & 1

& $S_{4,3,2^2}F^*_3\otimes S_{5^3,4^2}F_1$ &(4,3,2,2) & (5,5,5,4,4) & 7  \\ \hline

& $S_{4,3,2^2}F^*_3\otimes S_{6,5,4^3}F_1$ &(4,3,2,2) & (6,5,4,4,4) & 3

& $S_{4,3,2^2}F^*_3\otimes S_{5^4,3}F_1$ &(4,3,2,2) & (5,5,5,5,3) & 2 \\ \hline

& $S_{4,3,2^2}F^*_3\otimes S_{6,5^2,4,3}F_1$ &(4,3,2,2) & (6,5,5,4,3) & 1

& $S_{3^3,2}F^*_3\otimes S_{5^3,4^2}F_1$ &(3,3,3,2) & (5,5,5,4,4) & 9  \\ \hline

& $S_{3^3,2}F^*_3\otimes S_{6,5,4^3}F_1$ &(3,3,3,2) & (6,5,4,4,4) & 6 

& $S_{3^3,2}F^*_3\otimes S_{5^4,3}F_1$ &(3,3,3,2) & (5,5,5,5,3) & 5  \\ \hline

& $S_{3^3,2}F^*_3\otimes S_{6,5^2,4,3}F_1$ &(3,3,3,2) & (6,5,5,4,3) & 3

& $S_{3^3,2}F^*_3\otimes S_{6^2,4^2,3}F_1$ &(3,3,3,2) & (6,6,4,4,3) & 1  \\ \hline

 & $S_{3^3,2}F^*_3\otimes S_{7,4^4}F_1$ &(3,3,3,2) & (7,4,4,4,4) & 1 &&&& \\ \hline

{\bf 12} & $S_{5,3,2^2}F^*_3\otimes S_{5^5}F_1$ &(5,3,2,2) & (5,5,5,5,5) & 2 

& $S_{5,3,2^2}F^*_3\otimes S_{6,5^3,4}F_1$ &(5,3,2,2) & (6,5,5,5,4) & 1 \\ \hline

 & $S_{4^2,3,1}F^*_3\otimes S_{5^5}F_1$ &(4,4,3,1) & (5,5,5,5,5) & 2 

 & $S_{4^2,3,1}F^*_3\otimes S_{6,5^3,4}F_1$ &(4,4,3,1) & (6,5,5,5,4) & 1 \\ \hline

& $S_{4^2,2^2}F^*_3\otimes S_{5^5}F_1$ &(4,4,2,2) & (5,5,5,5,5) & 4 

& $S_{4^2,2^2}F^*_3\otimes S_{6,5^3,4}F_1$ &(4,4,2,2) & (6,5,5,5,4) & 2 \\ \hline

& $S_{4^2,2^2}F^*_3\otimes S_{6^2,5,4^2}F_1$ &(4,4,2,2) & (6,6,5,4,4) & 1

& $S_{4,3^2,2}F^*_3\otimes S_{6,5^3,4}F_1$ &(4,3,3,2) & (6,5,5,5,4) & 9  \\ \hline

& $S_{4,3^2,2}F^*_3\otimes S_{5^5}F_1$ &(4,3,3,2) & (5,5,5,5,5) & 7 

& $S_{4,3^2,2}F^*_3\otimes S_{6^2,5,4^2}F_1$ &(4,3,3,2) & (6,6,5,4,4) & 3  \\ \hline

& $S_{4,3^2,2}F^*_3\otimes S_{7,5^2,4^2}F_1$ &(4,3,3,2) & (7,5,5,4,4) & 1 

& $S_{4,3^2,2}F^*_3\otimes S_{6^2,5^2,3}F_1$ &(4,3,3,2) & (6,6,5,5,3) & 1  \\ \hline

& $S_{3^4}F^*_3\otimes S_{6,5^3,4}F_1$ &(3,3,3,3) & (6,5,5,5,4) & 8 

& $S_{3^4}F^*_3\otimes S_{6^2,5,4^2}F_1$ &(3,3,3,3) & (6,6,5,4,4) & 4  \\ \hline

& $S_{3^4}F^*_3\otimes S_{5^5}F_1$ &(3,3,3,3) & (5,5,5,5,5) & 4 

& $S_{3^4}F^*_3\otimes S_{7,5^2,4^2}F_1$ &(3,3,3,3) & (7,5,5,4,4) & 2  \\ \hline

& $S_{3^4}F^*_3\otimes S_{6^2,5^2,3}F_1$ &(3,3,3,3) & (6,6,5,5,3) & 2 

& $S_{3^4}F^*_3\otimes S_{7,6,4^3}F_1$ &(3,3,3,3) & (7,6,4,4,4) & 1 \\ \hline

 & $S_{3^4}F^*_3\otimes S_{6^3,4,3}F_1$ &(3,3,3,3) & (6,6,6,4,3) & 1 &&&& \\ \hline

\end{tabular}
\end{center}

\pagebreak

\subsection{ $E_8$ graded by $\alpha_3$} This format is $(1,7,8,2)$.

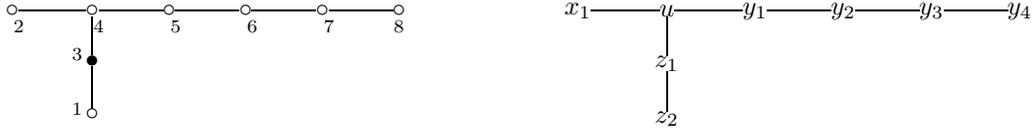
\[
\raisebox{1em}{\xymatrix@R=3ex{ 
*{\circ}<3pt> \ar@{-}[r]_<{2}  &*{\circ}<3pt>
\ar@{-}[r]_<{4}  &
  *{\circ}<3pt> \ar@{-}[r]_<{5} & *{\circ}<3pt>
\ar@{-}[r]_<{6} & *{\circ}<3pt>
\ar@{-}[r]_<{7} &*{\circ}<3pt> \ar@{}[]_<{8}
\\ & *{\bullet}<3pt> \ar@{-}[u]^<{3} \\ & *{\circ}<3pt> \ar@{-}[u]^<{1} }}
\hspace{2cm}
\raisebox{1em}{\xymatrix@R=3ex{ 
*{x_1}<3pt> \ar@{-}[r]_<{}  &*{u}<3pt>
\ar@{-}[r]_<{}  &
  *{y_1}<3pt> \ar@{-}[r]_<{} & *{y_2}<3pt>
\ar@{-}[r]_<{} & *{y_3}<3pt>
\ar@{-}[r]_<{} &*{y_4}<3pt> \ar@{}[]_<{}
\\ & *{z_1}<3pt> \ar@{-}[u]^<{} \\ & *{z_2}<3pt> \ar@{-}[u]^<{} }}
\]

\subsubsection{ $V(\omega_8)$}
There are 9 graded components.

\begin{center}

\begin{tabular}{|c||l|ll|c||l|ll|c|}
\hline
{\bf 0} & $F_1^*$ &(0,0)	&	(0,0,0,0,0,0,-1)	&	1 &&&&\\ \hline
{\bf 1} & $F_3^*\otimes F_1$ &(1,0)	&	(1,0,0,0,0,0,0)	&	1 &&&&\\ \hline
{\bf 2} & $\bigwedge^2 F_3^*\otimes\bigwedge^3F_1$ &(1,1)	&	(1,1,1,0,0,0,0)	&	1 &&&&\\ \hline
{\bf 3} & $S_{2,1}F_3^*\otimes \bigwedge^5F_1$ &(2,1)	&	(1,1,1,1,1,0,0)	&	1 &&&& \\ \hline
{\bf 4} & $S_{3,1}F^*_3\otimes\bigwedge^7 F_1$ &(3,1)	&	(1,1,1,1,1,1,1)	&	1 
& $S_{2,2}F^*_3\otimes S_{2,1^5}F_1$ &(2,2)	&	(2,1,1,1,1,1,0)	&	1 \\ \hline
 & $S_{2,2}F^*_3\otimes\bigwedge^7F_1$ &(2,2)	&	(1,1,1,1,1,1,1)	&	1&&&&  \\ \hline
{\bf 5} & $S_{3,2}F^*_3\otimes S_{2^2,1^5}F_1$ &(3,2)	&	(2,2,1,1,1,1,1)	&	1 &&&&\\ \hline
{\bf 6} & $S_{3,3}F^*_3\otimes S_{2^4,1^3}F_1$ &(3,3)	&	(2,2,2,2,1,1,1)	&	1 &&&&\\ \hline
{\bf 7} & $S_{4,3}F^*_3\otimes S_{2^6,1}F_1$ &(4,3)	&	(2,2,2,2,2,2,1)	&	1 &&&& \\ \hline
{\bf 8} & $S_{4,4}F^*_3\otimes S_{3,2^6}F_1$ &(4,4)	&	(3,2,2,2,2,2,2)	&	1 &&&&\\ \hline
\end{tabular}

\end{center}

\subsubsection{ $V(\omega_1)$}
There are 15 graded components.

\begin{center}

\begin{tabular}{|c||l|ll|c||l|ll|c|}
\hline
{\bf 0} & $F_3$ &(0,-1)	&	(0,0,0,0,0,0,0)	&	1 &&&& \\ \hline
{\bf 1} & $\bigwedge^2 F_1$&(0,0)	&	(1,1,0,0,0,0,0)	&	1 & & &&\\ \hline
{\bf 2} & $F_3^*\otimes \bigwedge^4 F_1$&(1,0)	&	(1,1,1,1,0,0,0)	&	1& &  & &\\ \hline
{\bf 3} & $S_2 F_3^*\otimes\bigwedge^6F_1$ & (2,0)	&	(1,1,1,1,1,1,0)	&	1 
& $\bigwedge^2 F_3^*\otimes S_{2,1^4} F_1$& (1,1)	&	(2,1,1,1,1,0,0)	&	1 \\ \hline
&  $\bigwedge^2 F_3^*\otimes \bigwedge^6 F_1$ &(1,1)	&	(1,1,1,1,1,1,0)	&	1 &&&&\\ \hline
{\bf 4} &$S_{2,1}F_3^*\otimes S_{2,1^6}F_1$ &(2,1)	&	(2,1,1,1,1,1,1)	&	2 &$S_{2,1}F_3^*\otimes S_{2,2,1,1,1,1}F_1$ &(2,1)	&	(2,2,1,1,1,1,0)	&	1 \\ \hline

{\bf 5} &$S_{3,1}F^*_3\otimes S_{2^3,1^4}F_1$ &(3,1)	&	(2,2,2,1,1,1,1)	&	1

&$S_{2,2}F_3^*\otimes S_{2^3,1^4}F_1$ &(2,2)	&	(2,2,2,1,1,1,1)	&	1 \\ \hline

&$S_{2,2}F_3^*\otimes S_{3,2,1^5}F_1$ &(2,2)	&	(3,2,1,1,1,1,1)	&	1

&$S_{2,2}F_3^*\otimes S_{2,2,2,2,1,1}F_1$ &(2,2)	&	(2,2,2,2,1,1,0)	&	1 \\ \hline

{\bf 6} 
&$S_{3,2}F_3^*\otimes S_{2^5,1^2}F_1$ &(3,2)	&	(2,2,2,2,2,1,1)	&	2

&$S_{3,2}F_3^*\otimes S_{3,2^3,1^3}F_1$ &(3,2)	&	(3,2,2,2,1,1,1)	&	1 \\ \hline

&$S_{3,2}F_3^*\otimes S_{2,2,2,2,2,2}F_1$&(3,2)	&	(2,2,2,2,2,2,0)	&	1
& & & &\\ \hline

{\bf 7}  & $S_{4,2}F^*_3\otimes S_{2^7}F_1$&(4,2)	&	(2,2,2,2,2,2,2)	&	1 &$S_{4,2}F^*_3\otimes S_{3,2^5,1}F_1$ & (4,2)	&	(3,2,2,2,2,2,1)	&	1  \\ \hline

& $S_{3,3}F^*_3\otimes S_{2^7}F_1$&(3,3)	&	(2,2,2,2,2,2,2)	&	2
 &$S_{3,3}F^*_3\otimes S_{3,2^5,1}F_1$ & (3,3)	&	(3,2,2,2,2,2,1)	&	2 \\ \hline

&$S_{3,3}F^*_3\otimes S_{3^2,2^3,1^2}F_1$ &(3,3)	&	(3,3,2,2,2,1,1)	&	1&&&&\\ \hline

{\bf 8} 
& $S_{4,3}F^*_3\otimes  S_{3^2,2^5}F_1$&(4,3)	&	(3,3,2,2,2,2,2)	&	2
 &$S_{4,3}F^*_3\otimes S_{4,2^6}F_1$ & (4,3)	&	(4,2,2,2,2,2,2)	&	1 \\ \hline
 
 &$S_{4,3}F^*_3\otimes S_{3^3,2^3,1}F_1$ &(4,3)	&	(3,3,3,2,2,2,1)	&	1& &&& \\ \hline

{\bf 9} & $S_{5,3}F^*_3\otimes S_{3^4,2^3}F_1 $&(5,3)	&	(3,3,3,3,2,2,2)	&	1
 &$S_{4,4}F^*_3\otimes S_{3^4,2^3}F_1$ & (4,4)	&	(3,3,3,3,2,2,2)	&	1\\ \hline
 & $S_{4,4}F^*_3\otimes S_{4,3^2,2^4}F_1$&(4,4)	&	(4,3,3,2,2,2,2)	&	1
 &$S_{4,4}F^*_3\otimes S_{3^5,2,1}F_1$ & (4,4)	&	(3,3,3,3,3,2,1)	&	1\\ \hline

{\bf 10} &$S_{5,4}F^*_3\otimes S_{3^6,2}F_1$&(5,4)	&	(3,3,3,3,3,3,2)	&	2
&$S_{5,4}F^*_3\otimes S_{4,3^4,2^2}F_1$&(5,4)	&	(4,3,3,3,3,2,2)	&	1\\ \hline

{\bf 11} & $S_{6,4}F^*_3\otimes S_{4,3^6}F_1$&(6,4)	&	(4,3,3,3,3,3,3)	&	1 &$S_{5,5}F^*_3\otimes S_{4^2,3^4,2}F_1$ & (5,5)	&	(4,4,3,3,3,3,2)	&	1  \\ \hline
&$S_{5,5}F^*_3\otimes S_{4,3^6}F_1$ &(5,5)	&	(4,3,3,3,3,3,3)	&	1 &&&&\\ \hline

{\bf 12} &$S_{6,5}F^*_3\otimes S_{4^3,3^4}F_1$ &(6,5)	&	(4,4,4,3,3,3,3)	&	1& &&& \\ \hline
{\bf 13} &$S_{6,6}F^*_3\otimes S_{4^5,3^2}F_1$&(6,6)	&	(4,4,4,4,4,3,3)	&	1& & &&  \\ \hline
{\bf 14} &$S_{7,6}F^*_3\otimes S_{4^7}F_1$&(7,6)	&	(4,4,4,4,4,4,4)	&	1& & & &\\ \hline
\end{tabular}

\end{center}

%
%

\pagebreak

\subsubsection{ $V(\omega_2)$}
There are  21 graded components. We exhibit the first 11. The others can be found by duality.

\begin{center}

\begin{tabular}{|c||l|ll|c||l|ll|c|}
\hline
{\bf 0} & $F_1$ &(0,0)  & (1,0,0,0,0,0,0) & 1 &&&& \\ \hline

{\bf 1}&$F_3^*\otimes\bigwedge^3 F_1$&(1,0) & (1,1,1,0,0,0,0) & 1 & & && \\ \hline

{\bf 2} &$\bigwedge^2 F_3^*\otimes\bigwedge^5 F_1$&(2,0)  & (1,1,1,1,1,0,0) & 1

&$\bigwedge^2 F_3^*\otimes S_{2,1,1,1}F_1$&(1,1)  & (2,1,1,1,0,0,0) & 1\\ \hline

&$S_2F_3^*\otimes\bigwedge^5F_1$&(1,1)  & (1,1,1,1,1,0,0) & 1& & & &\\ \hline

{\bf 3} &$S_3 F^*_3\otimes \bigwedge^7 F_1 $&(3,0)  & (1,1,1,1,1,1,1) & 1

&$S_{2,1}F_3^*\otimes S_{2,1^5}F_1$& (2,1) & (2,1,1,1,1,1,0) & 2 \\ \hline

&$S_{2,1}F_3^*\otimes \bigwedge^7 F_1$&(2,1)  & (1,1,1,1,1,1,1) & 2 

&$S_{2,1}F_3^*\otimes S_{2,2,1,1,1}F_1$& (2,1) & (2,2,1,1,1,0,0) & 1 \\ \hline

{\bf 4} &$S_{3,1}F^*_3\otimes S_{2^2,1^5}F_1$&(3,1)  & (2,2,1,1,1,1,1) & 2

& $S_{3,1}F^*_3\otimes S_{2^3,1^3}F_1$&(3,1) & (2,2,2,1,1,1,0) & 1 \\ \hline

&$S_{2,2}F_3^*\otimes S_{2^2,1^5}F_1$&(2,2) & (2,2,1,1,1,1,1) & 3 

&$S_{2,2}F_3^*\otimes S_{3,2,1^4}F_1$ & (2,2) & (3,2,1,1,1,1,0) & 1 \\ \hline

&$S_{2,2}F_3^*\otimes S_{3,1^6}F_1$&(2,2)  & (3,1,1,1,1,1,1) & 1

&$S_{2,2}F_3^*\otimes S_{2,2,2,2,1}F_1$&(2,2)  & (2,2,2,2,1,0,0) & 1\\ \hline

&$S_{2,2}F_3^*\otimes S_{2^3,1^3}F_1$&(2,2)  & (2,2,2,1,1,1,0) & 1& & & &\\ \hline

{\bf 5} &$S_{4,1}F^*_3\otimes S_{2^4,1^3}F_1$&(4,1) & (2,2,2,2,1,1,1) & 1

& $S_{3,2}F_3^*\otimes S_{2^4,1^3}F_1$& (3,2) & (2,2,2,2,1,1,1) & 4 \\ \hline

&$S_{3,2}F_3^*\otimes S_{3,2^2,1^4}F_1$&(3,2) & (3,2,2,1,1,1,1) & 2

& $S_{3,2}F_3^*\otimes S_{2^5,1}F_1$& (3,2) & (2,2,2,2,2,1,0) & 2 \\ \hline

&$S_{3,2}F_3^*\otimes S_{3^2,1^5}F_1$&(3,2) & (3,3,1,1,1,1,1) & 1

& $S_{3,2}F_3^*\otimes S_{3,2^3,1^2}F_1$& (3,2) & (3,2,2,2,1,1,0) & 1 \\ \hline

{\bf 6} &$S_{4,2}F_3^*\otimes S_{2^6,1}F_1$&(4,2) & (2,2,2,2,2,2,1) & 4

& $S_{4,2}F_3^*\otimes S_{3,2^4,1^2}F_1$&(4,2) & (3,2,2,2,2,1,1) & 2 \\ \hline

&$S_{4,2}F_3^*\otimes S_{3^2,2^2,1^3}F_1$&(4,2) & (3,3,2,2,1,1,1) & 1

& $S_{4,2}F_3^*\otimes S_{3,2^5}F_1$& (4,2)  & (3,2,2,2,2,2,0) & 1\\ \hline

&$S_{3,3}F_3^*\otimes S_{3,2^4,1^2}F_1$&(3,3) & (3,2,2,2,2,1,1) & 4

& $S_{3,3}F_3^*\otimes S_{2^6,1}F_1$& (3,3) & (2,2,2,2,2,2,1) & 4\\ \hline

&$S_{3,3}F_3^*\otimes S_{3^2,2^2,1^3}F_1$&(3,3) & (3,3,2,2,1,1,1) & 2

& $S_{3,3}F_3^*\otimes S_{4,2^3,1^3}F_1$& (3,3) & (4,2,2,2,1,1,1) & 1 \\ \hline

&$S_{3,3}F_3^*\otimes S_{3^2,2^3,1} F_1$&(3,3) & (3,3,2,2,2,1,0) & 1

& $S_{3,3}F_3^*\otimes  S_{3,2^5}F_1$& (3,3) & (3,2,2,2,2,2,0) & 1 \\ \hline

{\bf 7}&$S_{5,2}F^*_3\otimes S_{3^2,2^4,1}F_1$&(5,2) & (3,3,2,2,2,2,1) & 1

& $S_{5,2}F^*_3\otimes S_{3,2^6}F_1$&(5,2) & (3,2,2,2,2,2,2) & 1 \\ \hline

&$S_{4,3} F_3^*\otimes S_{3^2,2^4,1}F_1$&(4,3) & (3,3,2,2,2,2,1) & 6

& $S_{4,3} F_3^*\otimes S_{3,2^6}F_1$& (4,3) & (3,2,2,2,2,2,2) & 6\\ \hline

&$S_{4,3} F_3^*\otimes S_{4,2^5,1}F_1$&(4,3) & (4,2,2,2,2,2,1) & 2

& $S_{4,3} F_3^*\otimes S_{3^3,2^2,1^2}F_1$& (4,3) & (3,3,3,2,2,1,1) & 2\\ \hline

&$S_{4,3} F_3^*\otimes S_{4,3,2^3,1^2}F_1$&(4,3) & (4,3,2,2,2,1,1) & 1

& $S_{4,3} F_3^*\otimes S_{3^4,1^3}F_1$& (4,3) & (3,3,3,3,1,1,1) & 1 \\ \hline

&$S_{4,3} F_3^*\otimes S_{3^3,2^3} F_1$&(4,3) & (3,3,3,2,2,2,0) & 1 & & && \\ \hline

{\bf 8}&$S_{5,3}F^*_3\otimes S_{3^3,2^4}F_1$&(5,3) & (3,3,3,2,2,2,2) & 4

& $S_{5,3}F^*_3\otimes S_{3^4,2^2,1}F_1$&(5,3) & (3,3,3,3,2,2,1) & 3 \\ \hline

&$S_{5,3}F^*_3\otimes S_{4,3,2^5}F_1$&(5,3) & (4,3,2,2,2,2,2) & 2

& $S_{5,3}F^*_3\otimes S_{4,3^2,2^3,1}F_1$& (5,3) & (4,3,3,2,2,2,1) & 1\\ \hline

&$S_{4,4}F_3^*\otimes S_{3^3,2^4}F_1$&(4,4) & (3,3,3,2,2,2,2) & 5

& $S_{4,4}F_3^*\otimes S_{4,3,2^5}F_1$& (4,4) & (4,3,2,2,2,2,2) & 4\\ \hline

&$S_{4,4}F_3^*\otimes S_{3^4,2^2,1}F_1$&(4,4) & (3,3,3,3,2,2,1) & 4

& $S_{4,4}F_3^*\otimes S_{4,3^2,2^3,1}F_1$& (4,4) & (4,3,3,2,2,2,1) & 2 \\ \hline

&$S_{4,4}F_3^*\otimes S_{5,2^6}F_1$&(4,4) & (5,2,2,2,2,2,2) & 1

& $S_{4,4}F_3^*\otimes S_{4^2,2^4,1}F_1$& (4,4) & (4,4,2,2,2,2,1) & 1 \\ \hline

&$S_{4,4}F_3^*\otimes S_{4,3^3,2,1^2}F_1$&(4,4) & (4,3,3,3,2,1,1) & 1

& $S_{4,4}F_3^*\otimes S_{3^5,2}F_1$& (4,4) & (3,3,3,3,3,2,0) & 1 \\ \hline

&$S_{4,4}F_3^*\otimes S_{3^5,1^2}F_1$&(4,4) & (3,3,3,3,3,1,1) & 1 & & && \\ \hline

{\bf 9} &$S_{6,3}F^*_3\otimes S_{4,3^3,2^3}F_1$&(6,3) & (4,3,3,3,2,2,2) & 1

& $S_{6,3}F^*_3\otimes S_{3^6,1}F_1$&(6,3) & (3,3,3,3,3,3,1) & 1\\ \hline

&$S_{6,3}F^*_3\otimes S_{3^5,2^2}F_1$&(6,3) & (3,3,3,3,3,2,2) & 1 

&$S_{5,4}F^*_3\otimes S_{3^5,2^2}F_1$ & (5,4) & (3,3,3,3,3,2,2) & 7\\ \hline

&$S_{5,4}F^*_3\otimes S_{4,3^3,2^3}F_1$&(5,4) & (4,3,3,3,2,2,2) & 6

& $S_{5,4}F^*_3\otimes S_{3^6,1}F_1$&(5,4) & (3,3,3,3,3,3,1) & 4 \\ \hline

&$S_{5,4}F^*_3\otimes S_{4,3^4,2,1}F_1$&(5,4) & (4,3,3,3,3,2,1) & 3 

&$S_{5,4}F^*_3\otimes S_{4^2,3,2^4}F_1$ & (5,4) & (4,4,3,2,2,2,2) & 2 \\ \hline

&$S_{5,4}F^*_3\otimes S_{5,3^2,2^4}F_1$&(5,4) & (5,3,3,2,2,2,2) & 1

& $S_{5,4}F^*_3\otimes S_{4^2,3^2,2^2,1}F_1$&(5,4) & (4,4,3,3,2,2,1) & 1 \\ \hline

{\bf 10} &$S_{6,4}F^*_3\otimes S_{4,3^5,2}F_1$&(6,4) & (4,3,3,3,3,3,2) & 6

& $S_{6,4}F^*_3\otimes S_{3^7}F_1$&(6,4) & (3,3,3,3,3,3,3) & 3 \\ \hline

&$S_{6,4}F^*_3\otimes S_{4^2,3^3,2^2}F_1$&(6,4) & (4,4,3,3,3,2,2) & 2

&$S_{6,4}F^*_3\otimes S_{5,3^4,2^2}F_1$ & (6,4) & (5,3,3,3,3,2,2) & 1 \\ \hline

&$S_{6,4}F^*_3\otimes S_{4^3,3,2^3}F_1$&(6,4) & (4,4,4,3,2,2,2) & 1

& $S_{6,4}F^*_3\otimes S_{4^2,3^4,1}F_1$&(6,4) & (4,4,3,3,3,3,1) & 1 \\ \hline

&$S_{5,5}F^*_3\otimes S_{4,3^5,2}F_1$&(5,5) & (4,3,3,3,3,3,2) & 7

&$S_{5,5}F^*_3\otimes S_{4^2,3^3,2^2}F_1$ & (5,5) & (4,4,3,3,3,2,2) & 5 \\ \hline

&$S_{5,5}F_3^*\otimes S_{3^7}F_1 $&(5,5) & (3,3,3,3,3,3,3) & 3

& $S_{5,5}F^*_3\otimes S_{5,3^4,2^2}F_1$&(5,5) & (5,3,3,3,3,2,2) & 2 \\ \hline

&$S_{5,5}F^*_3\otimes S_{4^2,3^4,1}F_1$&(5,5) & (4,4,3,3,3,3,1) & 2

&$S_{5,5}F^*_3\otimes S_{5,4,3^2,2^3}F_1$ & (5,5) & (5,4,3,3,2,2,2) & 1 \\ \hline

&$S_{5,5}F^*_3\otimes S_{4^3,3^2,2,1}F_1$&(5,5) & (4,4,4,3,3,2,1) & 1

& $S_{5,5}F^*_3\otimes S_{4^3,3,2^3}F_1$&(5,5) & (4,4,4,3,2,2,2) & 1 \\ \hline

\end{tabular}

\end{center}


\subsection{ $E_8$ graded by $\alpha_2$} This format is $(2,8,7,1)$.

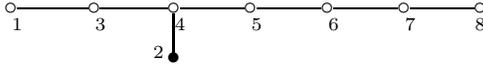
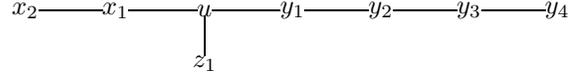
\[
\raisebox{1em}{\xymatrix@R=3ex{ *{\circ}<3pt> \ar@{-}[r]_<{1}  &
*{\circ}<3pt> \ar@{-}[r]_<{3}  &*{\circ}<3pt>
\ar@{-}[r]_<{4}  &
  *{\circ}<3pt> \ar@{-}[r]_<{5} & *{\circ}<3pt>
\ar@{-}[r]_<{6}  & *{\circ}<3pt>
\ar@{-}[r]_<{7} &*{\circ}<3pt> \ar@{}[]_<{8}
\\ & & *{\bullet}<3pt> \ar@{-}[u]^<{2}  }}
\hspace{2cm}
\raisebox{1em}{\xymatrix@R=3ex{ *{x_2}<3pt> \ar@{-}[r]_<{}  &
*{x_1}<3pt> \ar@{-}[r]_<{}  &*{u}<3pt>
\ar@{-}[r]_<{}  &
  *{y_1}<3pt> \ar@{-}[r]_<{} & *{y_2}<3pt>
\ar@{-}[r]_<{}  & *{y_3}<3pt>
\ar@{-}[r]_<{} &*{y_4}<3pt> \ar@{}[]_<{}
\\ & & *{z_1}<3pt> \ar@{-}[u]^<{}  }}
\]

\subsubsection{ $V(\omega_8)$}
There are 7 graded components.

\begin{center}

\begin{tabular}{|c||l|ll|c||l|ll|c|}
\hline
{\bf 0} & $F_1^*$ &&  (0,0,0,0,0,0,0,-1)  & 1 &&&&\\ \hline
{\bf 1} & $\bigwedge^2 F_1$ &&  (1,1,0,0,0,0,0,0) & 1 &&&&\\ \hline
{\bf 2} & $\bigwedge^5 F_1$ &&  (1,1,1,1,1,0,0,0) & 1 &&&&\\ \hline
{\bf 3} & $S_{2,1^6} F_1$ &&  (2,1,1,1,1,1,1,0) & 1 & $\bigwedge^8 F_1$ &&  (1,1,1,1,1,1,1,1) & 1 \\ \hline
{\bf 4} & $S_{2^3,1^5} F_1$ &&  (2,2,2,1,1,1,1,1) & 1 &&&&\\ \hline
{\bf 5} & $S_{2^6,1^2} F_1$ &&  (2,2,2,2,2,2,1,1) & 1  &&&&\\ \hline
{\bf 6} & $S_{3,2^7} F_1$ &&  (3,2,2,2,2,2,2,2) & 1 &&&&\\ \hline
\end{tabular}

\end{center}

\subsubsection{ $V(\omega_1)$}
There are 11 graded components.

\begin{center}

\begin{tabular}{|c||l|ll|c||l|ll|c|}
\hline
{\bf 0} & $F_1$ &&  (1,0,0,0,0,0,0,0) & 1 &&&& \\ \hline
{\bf 1} & $\bigwedge^4 F_1$ &&  (1,1,1,1,0,0,0,0) & 1 & & &&\\ \hline
{\bf 2} & $S_{2,1^5} F_1$ &&  (2,1,1,1,1,1,0,0) & 1
& $\bigwedge^7 F_1$ &&  (1,1,1,1,1,1,1,0) & 1\\ \hline
{\bf 3} & $S_{3,1^7} F_1$ &&  (3,1,1,1,1,1,1,1) & 1
& $S_{2^3,1^4} F_1$ &&  (2,2,2,1,1,1,1,0) & 1  \\ \hline
& $S_{2^2,1^6} F_1$ &&  (2,2,1,1,1,1,1,1) & 1 &&&&\\ \hline
{\bf 4} & $S_{3,2^3,1^4} F_1$ &&  (3,2,2,2,1,1,1,1) & 1
& $S_{2^6,1} F_1$ &&  (2,2,2,2,2,2,1,0) & 1  \\ \hline
& $S_{2^5,1^3} F_1$ &&  (2,2,2,2,2,1,1,1) & 1 &&&&\\ \hline
{\bf 5} & $S_{3^2,2^4,1^2} F_1$ &&  (3,3,2,2,2,2,1,1) & 1
& $S_{3,2^6,1} F_1$ &&  (3,2,2,2,2,2,2,1) & 2  \\ \hline
& $S_{2^8} F_1$ &&  (2,2,2,2,2,2,2,2) & 1 &&&&\\ \hline
{\bf 6} & $S_{4,3,2^6} F_1$ &&  (4,3,2,2,2,2,2,2) & 1
& $S_{3^4,2^3,1} F_1$ &&  (3,3,3,3,2,2,2,1) & 1 \\ \hline
& $S_{3^3,2^5} F_1$ &&  (3,3,3,2,2,2,2,2) & 1 &&&&\\ \hline
{\bf 7} & $S_{4,3^4,2^3} F_1$ &&  (4,3,3,3,3,2,2,2) & 1
& $S_{3^7,1} F_1$ &&  (3,3,3,3,3,3,3,1) & 1  \\ \hline
& $S_{3^6,2^2} F_1$ &&  (3,3,3,3,3,3,2,2) & 1 &&&&\\ \hline
{\bf 8} & $S_{4^2,3^5,2} F_1$ &&  (4,4,3,3,3,3,3,2) & 1
& $S_{4,3^7} F_1$ &&  (4,3,3,3,3,3,3,3) & 1 \\ \hline
{\bf 9} & $S_{4^4,3^4} F_1$ &&  (4,4,4,4,3,3,3,3) & 1 & & &&  \\ \hline
{\bf 10} & $S_{4^7,3} F_1$ &&  (4,4,4,4,4,4,4,3) & 1& & & &\\ \hline
\end{tabular}

\end{center}

\pagebreak

\subsubsection{ $V(\omega_2)$}
There are  17 graded components. We present the first 9. The others can be filled by duality.

\begin{center}

\begin{tabular}{|c||l|ll|c||l|ll|c|}
\hline
{\bf 0} &	$\mathbb C$	&&	(0,0,0,0,0,0,0,0)	&	1 &&&& \\ \hline
{\bf 1}& $\bigwedge^3 F_1$		&&	(1,1,1,0,0,0,0,0)	&	1  & & && \\ \hline
{\bf 2}&	$S_{2,1^4}F_1$	&&	(2,1,1,1,1,0,0,0)	&	1
&	$\bigwedge^6 F_1$		&&	(1,1,1,1,1,1,0,0)	&	1\\ \hline

{\bf 3} &	$S_{3,1^6}F_1$	&&	(3,1,1,1,1,1,1,0)	&	1
&	$S_{2^3,1^3}F_1$	&&	(2,2,2,1,1,1,0,0)	&	1 \\ \hline
&	$S_{2^2,1^5}F_1$	&&	(2,2,1,1,1,1,1,0)	&	1
&	$S_{2,1^7}F_1$	&&	(2,1,1,1,1,1,1,1)	&	2 \\ \hline

{\bf 4}&	$S_{3,2^3,1^3}F_1$	&&	(3,2,2,2,1,1,1,0)	&	1
&	$S_{3,2^2,1^5}F_1$	&&	(3,2,2,1,1,1,1,1)	&	2 \\ \hline
&	$S_{2^6}F_1$	&&	(2,2,2,2,2,2,0,0)	&	1
&	$S_{2^5,1^2}F_1$	&&	(2,2,2,2,2,1,1,0)	&	1 \\ \hline
&	$S_{2^4,1^4}F_1$	&&	(2,2,2,2,1,1,1,1)	&	2  & & && \\ \hline

{\bf 5} &	$S_{4,2^4,1^3}F_1$	&&	(4,2,2,2,2,1,1,1)	&	1
&	$S_{3^3,2,1^4}F_1$	&&	(3,3,3,2,1,1,1,1)	&	1  \\ \hline
&	$S_{3^2,2^4,1}F_1$	&&	(3,3,2,2,2,2,1,0)	&	1
&	$S_{3^2,2^3,1^3}F_1$	&&	(3,3,2,2,2,1,1,1)	&	1  \\ \hline
&	$S_{3,2^6}F_1$	&&	(3,2,2,2,2,2,2,0)	&	1
&	$S_{3,2^5,1^2}F_1$	&&	(3,2,2,2,2,2,1,1)	&	4  \\ \hline
&	$S_{2^7,1}F_1$	&&	(2,2,2,2,2,2,2,1)	&	2  & & && \\ \hline

{\bf 6}&	$S_{4,3^2,2^3,1^2}F_1$	&&	(4,3,3,2,2,2,1,1)	&	1
&	$S_{4,3,2^5,1}F_1$	&&	(4,3,2,2,2,2,2,1)	&	2 \\ \hline
&	$S_{4,2^7}F_1$	&&	(4,2,2,2,2,2,2,2)	&	3
&	$S_{3^4,2^3}F_1$	&&	(3,3,3,3,2,2,2,0)	&	1 \\ \hline
&	$S_{3^4,2^2,1^2}F_1$	&&	(3,3,3,3,2,2,1,1)	&	2
&	$S_{3^3,2^4,1}F_1$	&&	(3,3,3,2,2,2,2,1)	&	3  \\ \hline
&	$S_{3^2,2^6}F_1$	&&	(3,3,2,2,2,2,2,2)	&	3  & & && \\ \hline

{\bf 7}&	$S_{5,3^2,2^5}F_1$	&&	(5,3,3,2,2,2,2,2)	&	1
&	$S_{4^2,3^2,2^3,1}F_1$	&&	(4,4,3,3,2,2,2,1)	&	1 \\ \hline
&	$S_{4^2,3,2^5}F_1$	&&	(4,4,3,2,2,2,2,2)	&	1
&	$S_{4,3^5,1^2}F_1$	&&	(4,3,3,3,3,3,1,1)	&	1\\ \hline
&	$S_{4,3^4,2^2,1}F_1$	&&	(4,3,3,3,3,2,2,1)	&	2
&	$S_{4,3^3,2^4}F_1$	&&	(4,3,3,3,2,2,2,2)	&	4  \\ \hline
&	$S_{3^7}F_1$	&&	(3,3,3,3,3,3,3,0)	&	1
&	$S_{3^6,2,1}F_1$	&&	(3,3,3,3,3,3,2,1)	&	3 \\ \hline
&	$S_{3^5,2^3}F_1$	&&	(3,3,3,3,3,2,2,2)	&	3  & & && \\ \hline

{\bf 8}&	$S_{5,4,3^3,2^3}F_1$	&&	(5,4,3,3,3,2,2,2)	&	1
&	$S_{5,3^5,2^2}F_1$	&&	(5,3,3,3,3,3,2,2)	&	2 \\ \hline
&	$S_{4^4,2^4}F_1$	&&	(4,4,4,4,2,2,2,2)	&	1
&	$S_{4^3,3^3,2,1}F_1$	&&	(4,4,4,3,3,3,2,1)	&	1 \\ \hline
&	$S_{4^3,3^2,2^3}F_1$	&&	(4,4,4,3,3,2,2,2)	&	1
&	$S_{4^2,3^5,1}F_1$	&&	(4,4,3,3,3,3,3,1)	&	2 \\ \hline
&	$S_{4^2,3^4,2^2}F_1$	&&	(4,4,3,3,3,3,2,2)	&	4
&	$S_{4,3^6,2}F_1$	&&	(4,3,3,3,3,3,3,2)	&	5  \\ \hline
&	$S_{3^8}F_1$	&&	(3,3,3,3,3,3,3,3)	&	1  & & && \\ \hline

\end{tabular}

\end{center}

}

\bigskip

\section{Some patterns and observations} \label{sec-observations}

In this section we observe some general patterns that can be deduced from the tables. They will form a basis for a conjecture.

\begin{prp} Let us restrict to the  formats of resolutions of cyclic modules, i.e. $p=2$.
Then the tensors giving the first graded components of three critical representations are:
\begin{enumerate}
\item $W_1(d_3)=F_2^*\otimes\bigwedge^2 F_1$ gives the tensor of multiplicative structure $F_1\otimes F_1\rightarrow F_2$ on $\FF_\bullet$,
\item $W_1(d_2)=F_2\otimes F^*_3\otimes  F_1$ gives the tensor of multiplicative structure $F_1\otimes F_2\rightarrow F_3$ on $\FF_\bullet$,
\item $W_1(a_2)=F_3^*\otimes\bigwedge^3 F_1$ gives the tensor of multiplicative structure $\bigwedge^3 F_1\rightarrow F_3$ on $\FF_\bullet$.
\end{enumerate}
\end{prp}
\begin{proof} The formulas follow immediately from the parabolic BGG complex. See \cite{Ku}, Section 9.2.
\end{proof}

For other formats we have a similar interpretation.

Consider the format with the ranks $(r_1, r_2, r_3)$. We have a comparison map from the Buchsbaum--Rim complex to the complex $\FF_\bullet$.
\[\begin{xy}
\xymatrix{
0 \ \ar[r]  &  \ \ \ \ \  F_3 \ \ \ \ \ \ar[r]^{d_3} & \ \ \ \ \ F_2 \ \ \ \ \ \ar[r]^{d_2} & \ F_1 \ \ar[r]^{d_1} &  \ F_0 
\\
\cdots \ \ar[r]  &    \bigwedge^{r_1+2}F_1\otimes \otimes F^*_0\otimes\bigwedge^{r_1}F^*_0 \ar[r] \ar[u]_{v_3} & \bigwedge^{r_1+1}F_1\otimes \bigwedge^{r_1}F^*_0  \ar[r] \ar[u]_{v_1} & \ F_1 \ \ar[r]^{d_1}  \ar[u]_{=}&  \ F_0  \ar[u]_{=}
}
\end{xy}\]
Similarly we have a lifting
\[\begin{xy}
\xymatrix{
0 \ \ar[r]  &  \ \ \   F_3 \ \ \  \ar[r]^{d_3} & \ \ \ F_2 \ \ \  \ar[r]^{d_2} & \ F_1 \ \ar[r]^{d_1} &  \ F_0 
\\
 &       & \ar[ul]_{v_2} \bigwedge^{r_1}F_1\otimes F_2  \ar[u]_{q_1} & &  
}
\end{xy}\]
of the cycle $q_1=v_1d_1-(\bigwedge^{r_1}d_1\otimes d_2)$, which we denote $v_2$.

\begin{prp} Let us deal with a general format.
Then the tensors giving the first graded components of three critical representations are:
\begin{enumerate}
\item $W_1(d_3)=F_2^*\otimes\bigwedge^{r_1+1} F_1$ gives the tensor $v_1$ in the diagram above,
\item $W_1(d_2)=F_2\otimes F^*_3\otimes  \bigwedge^{r_1}F_1$ gives the tensor $v_2$ in the diagram above,
\item $W_1(a_2)=F_0^*\otimes F_3^*\otimes\bigwedge^{r_1+2} F_1$ gives the tensor $v_3$ in the diagram above.
\end{enumerate}
\end{prp}
\begin{proof} The formulas follow immediately from the parabolic BGG complex. See \cite{Ku}, Section 9.2.
\end{proof}

For the Dynkin formats with $p=2$, notice that, because three critical representations are finite dimensional, they have top graded components. Moreover, since critical representations are either self dual or (in the case of $D_n$ or $E_6$) possibly dual to each other,
we have the following. 

\begin{prp} Let us consider Dynkin format, with the exception of $(1,n-1,1)$ with $n$ odd.
Then the top components $W_{top}(d_3)$, $W_{top}(d_2)$, $W_{top}(a_2)$ give tensors which can be interpreted as the differentials in the complex
$$\FF^{top}_\bullet: F^*_3\buildrel{\partial_3}\over\rightarrow F_2\buildrel{\partial_2}\over\rightarrow F^*_1\buildrel{\partial_1}\over\rightarrow F_0,$$
except for the format $(1,4,2)$ when we get a complex
$$\FF^{top}_\bullet: F^*_3\rightarrow F^*_2\rightarrow F^*_1\rightarrow F_0.$$
\end{prp}

\begin{proof} This follows from a careful examination of the tables. The fact that we get such complexes follows from the decomposition of ${\hat R}_{gen}$ given in Proposition \ref{prp-decrgen} into irreducibles. One can see we do not have in ${\hat R}_{gen}$ representations which could be tensors giving compositions $\partial_2\partial_3$ and $\partial_1\partial_2$.
\end{proof}

The case of almost complete intersections (proved in \cite{CVW18}) is the basis for the following conjecture.

\begin{conj}[\cite{CVW18}] Assume we deal with Dynkin format. The open set $U_{CM}$ of points in $Spec({\hat R}_{gen})$ at which $\FF^{gen}_\bullet$ is a resolution of a Cohen-Macaulay module consists of points where the complex $\FF^{top}_\bullet$ is split exact.
\end{conj}


\section{Appendix} \label{sec-appendix}

In this appendix, we explain how one can generate the tables in the previous sections using SageMath \cite{sage}. Our method uses {\em crystals} introduced by M. Kashiwara \cite{Ka}. A systematic introduction to the theory of crystals can be found in \cite{HK}.  

As all the cases are similar, we only present a SageMath code for the table in Section \ref{E7V7}.

\begin{lstlisting}
sage: La=RootSystem(['E',7]).weight_space().fundamental_weights()
sage: B=crystals.LSPaths(La[7])
sage: for x in B:
....:     if x.is_highest_weight([1,2,3,4,6,7]):
....:         a=x.to_highest_weight()[1]
....:         l=len(a); d=0
....:         for j in range (0,l):
....:             if a[j]== 5:
....:                 d=d+1
....:         print d, ",", x.weight()
....:         
0 , Lambda[7]
1 , Lambda[4] - Lambda[5]
2 , Lambda[1] - Lambda[5] + Lambda[6]
3 , Lambda[2] - Lambda[5] + Lambda[7]
4 , Lambda[3] - Lambda[5]
5 , -Lambda[5] + Lambda[6]
\end{lstlisting}
Each line of the outcome shows the degree and highest weight of an irreducible component. One can convert a highest weight into a pair of partitions. For example, the highest weight 
\begin{center}
{\fontfamily{qcr}\selectfont
Lambda[1] - Lambda[5] + Lambda[6]
}
\end{center}
of degree $2$ corresponds to $(1,0,0) \ (1,1,1,1,0)$ as {\fontfamily{qcr}\selectfont Lambda[6]} becomes the first fundamental weight of $\mathfrak{gl}(3)$ and {\fontfamily{qcr}\selectfont Lambda[1]} the last fundamental weight of $\mathfrak{gl}(5)$. (See the Dynkin diagram in Section \ref{E7A5}.)

\bibliographystyle{amsplain}

\providecommand{\MR}[1]{\mbox{\href{http://www.ams.org/mathscinet-getitem?mr=#1}{#1}}}
\renewcommand{\MR}[1]{\mbox{\href{http://www.ams.org/mathscinet-getitem?mr=#1}{#1}}}
\providecommand{\href}[2]{#2}

\end{document}